\documentclass{amsart}

\usepackage{graphicx}
\usepackage{enumerate}
\usepackage{amssymb}

\newtheorem{Theor}{Theorem}
\newtheorem{Corol}{Corollary}

\newtheorem{theorem}{Theorem}[section]
\newtheorem{lemma}[theorem]{Lemma}
\newtheorem{proposition}[theorem]{Proposition}

\newtheorem{rem}[theorem]{Remark}

\newtheorem{fact}[theorem]{Fact}

\newtheorem*{Newhouse}{Motivation I}
\newtheorem*{Bonatti-Diaz}{Motivation II}

\theoremstyle{definition}
\newtheorem{definition}[theorem]{Definition}

\newtheorem{remark}[theorem]{Remark}

\numberwithin{equation}{section}

\def\II{{\mathbb I}} \def\JJ{{\mathbb J}}

 \def\NN{{\mathbb N}} 

 \def\RR{{\mathbb R}}

 \def\ZZ{{\mathbb Z}}

\def\mfT{{\mathfrak{T}}}

\def\Si{\Sigma}
\def\La{\Lambda}
\def\De{\Delta}

\def\Ga{\Gamma}

\def\la{\lambda}
\def\be{\beta}
\def\de{\delta}

\def\ve{\varepsilon}

\def\cC{{\mathcal C}}    \def\cU{{\mathcal
U}}
\def\cD{{\mathcal D}}

\def\st{{\operatorname{s}}}
\def\sst{{\operatorname{ss}}}

\def\ct{{\operatorname{c}}}
\def\cut{{\operatorname{cu}}}
\def\ut{{\operatorname{u}}}
\def\uut{{\operatorname{uu}}}

\def\loc{{\operatorname{loc}}}
\def\sign{{\operatorname{sign}}}
\def\scyc{{\operatorname{{\bf{sc}}}}}
\def\sscyc{{\operatorname{{\bf{ssc}}}}}

\def\sind{{\st\operatorname{-index}}}
\def\uind{{\ut\operatorname{-index}}}

\def\tpq{{{\tau_{p,q}}}}
\def\tqp{{{\tau_{q,p}}}}

\begin{document}

\title[Stabilization of heterodimensional cycles]{Stabilization of heterodimensional cycles}

\author{C. Bonatti}
\address{Institut de Math\'ematiques de Bourgogne, BP 47 870, 1078 Dijon Cedex, France}
\email{bonatti@u-bourgogne.fr}

\author{L. J. D\'iaz}
\address{Depto. Matem\'atica, PUC-Rio Marqu\^es de S. Vicente 225
22453-900 Rio de Janeiro RJ  Brazil} \email{lodiaz@mat.puc-rio.br}

\author{S. Kiriki}
\address{Department of Mathematics, Kyoto University of Education,
1 Fukakusa-Fujinomori, Fushimi, Kyoto, 612-8522, JAPAN}
\email{skiriki@kyokyo-u.ac.jp}

\subjclass[2000]{Primary:37C29, 37D20, 37D30}
\keywords{heterodimensional cycle, homoclinic class, hyperbolic
set, blender, $C^{1}$ robustness}

\date{\today}


\begin{abstract}
We consider diffeomorphisms $f$ with
 heteroclinic cycles associated to saddles $P$ and $Q$ of different indices.
 We say that a cycle of this type can be stabilized  if there are diffeomorphisms close to $f$ with a robust cycle associated to hyperbolic sets containing the continuations of $P$ and $Q$.
We focus on the case where the indices of these two saddles differ by one.
We prove that, excluding one particular case (so-called twisted cycles that additionally satisfy some geometrical restrictions), all such cycles can be stabilized.
\end{abstract}
\maketitle




\section{introduction}

In \cite{P08} Palis proposed a program whose main goal is a
geometrical description for the behavior of most dynamical
systems. This program pays special attention to the generation of
non-hyperbolic dynamics and to robust dynamical properties
(i.e., properties that hold for open sets of dynamical systems).
An important part of this program is the {\emph{Density Conjecture
 olicity versus cycles)}}: the two main sources of
non-hyperbolic dynamics are heterodimensional cycles and homoclinic
tangencies (shortly, {\emph{cycles}}), see \cite[Conjecture
1]{P08}\footnote{This conjecture was proved by Pujals-Sambarino for
$C^1$-surface diffeomorphisms in \cite{PS00} (since
heterodimensional cycles only can occur in manifolds of dimension
$n\ge 3$, for surface diffeomorphisms it is enough to consider
homoclinic tangencies).}. The goal of this paper is to study the
{\emph{generation of robust heterodimensional cycles}} (see
Definition~\ref{d.severalcycles}).
 
Besides Palis' program, we have the  following 
two motivations for this paper:

\begin{Newhouse}[{\cite{N79,PV94,R95}}]
Every $C^2$-diffeomorphism having a homoclinic
tangency associated with  a saddle $P$ is in the $C^{2}$-closure
of the set of diffeomorphisms having $C^{2}$-robust homoclinic
tangencies.
Moreover, these robust homoclinic tangencies can be taken
associated to hyperbolic sets containing the continuations of the
saddle $P$.
\end{Newhouse}

Using the terminology that will be introduced in this paper
this means that homoclinic tangencies of $C^2$-diffeomorphisms
can be {\emph{stabilized,}} see 
 Definition~\ref{d.severalcycles}. On the other hand,  
for $C^1$-diffeo\-mor\-phisms of surfaces
homoclinic tangencies
 cannot
be stabilized, see
in \cite{Gugu}. 
This leads to the following motivation.

\begin{Bonatti-Diaz}[{\cite{BDijmj}}]
Every diffeomorphism with a
 heterodimensional cycle associated with
a pair of hyperbolic saddles $P$ and $Q$ with $\dim E^s(P)
=\dim E^s(Q)\pm 1$ belongs to the $C^{1}$-closure of the
set of diffeomorphisms having $C^{1}$-robust heterodimensional
cycles. Here $E^s$ denotes the stable bundle of a saddle.
\end{Bonatti-Diaz}

One may think of 
the result in
Motivation II as a version of the results in Motivation I for
heterodimensional cycles in the $C^1$-setting. However,
the results in \cite{BDijmj} does not provide information about the
relation between the hyperbolic sets involved in the robust cycles
and the saddles in the initial one. Thus, one aims for
an extension of \cite{BDijmj} 
giving some information about the hyperbolic sets displaying the robust cycles,
see \cite[Question 1.9]{BDijmj}.

In this paper we prove that, with the exception of a special type
of heterodimensional cycles (so-called {\emph{twisted cycles}},
see Definition~\ref{d.twistenontwisted}), the hyperbolic sets
exhibiting the robust cycles can be taken containing the
continuations of the saddles in the initial cycle.  In fact,
by \cite{BDcontra} 
our results cannot be improved:
there are twisted cycles that cannot be stabilized, that is, the
hyperbolic sets with robust cycles cannot be taken containing the continuations of the saddles in the initial cycle.

To state precisely our results we need to introduce some
definitions. 
Recall that if $\La$ is a hyperbolic basic set of a
diffeomorphism $f\colon M\to M$ then there are a neighborhood $\cU_f$ of $f$ in
the space of $C^1$-diffeomorphisms
 and a continuous map $\cU_f\to M\colon g\mapsto \La_g$,
such that $\La_f=\La$, 
$\La_g$ is a hyperbolic basic set, and
the dynamics of $f|_{\La}$ and $g|_{\La_g}$ are conjugate.
 The
set $\La_g$ is called the {\emph{continuation}} of $\La$ for $g$.
Note that these continuations are uniquely defined.

\begin{definition}[Heterodimensional cycles] $\,$
\begin{itemize}
\item
The {\emph{$\st$-index}} ({\emph{$\ut$-index}}) of a
transitive hyperbolic set is the dimension of its stable (unstable) bundle.
\item
A diffeomorphism $f$ has a {\emph{heterodimensional cycle}}
associated to transitive hyperbolic basic sets $\Lambda$ and
$\Sigma$ of $f$ if these sets have different $\st$-indices and
their invariant manifolds meet cyclically, that is, if $W^\st(\Lambda,f)\cap
W^\ut(\Sigma,f)\ne\emptyset$ and $W^\ut(\Lambda,f)\cap
W^\st(\Sigma,f)\ne\emptyset$.
\item
The heterodimensional cycle has {\emph{coindex $k$}} if
 $\sind (\La)=\sind (\Si)\pm k$. In such a case we just write
 {\emph{coindex $k$ cycle.}}
\item
A diffeomorphism $f$ has a {\emph{$C^1$-robust heterodimensional
cycle}}  associated to its hyperbolic basic sets
$\La$ and $\Si$
if there is a $C^1$-neighborhood $\mathcal{U}$ of $f$
such that every diffeomorphism $g\in \mathcal{U}$ has a 
a heterodimensional cycle associated to the 
continuations
$\Lambda_g$ and $\Sigma_g$  of $\La$ and $\Si$, respectively.
\item
Consider a diffeomorphism $f$ with a heterodimensional cycle associated to a
 pair  of saddles $P$ and $Q$. This cycle  {\emph{can be $C^1$-stabilized}} if
every $C^1$-neighborhood $\cU$ of $f$ contains a diffeomorphism
$g$ with hyperbolic basic sets $\La_g\ni P_g$ and $\Si_g\ni Q_g$ having a robust heterodimensional cycle. Otherwise the cycle is said to be
{\emph{$C^1$-fragile.}}
\end{itemize}
\label{d.severalcycles}
\end{definition}

Remark that, by the Kupka-Smale genericity theorem (invariant
manifolds of hyperbolic periodic points of generic diffeomorphisms
are in general position), at least one of the hyperbolic sets involved in a robust 
cycle is necessarily {\emph{non-trivial,}} that is, not a periodic orbit.

\begin{definition}[Homoclinic class]
The {\emph{homoclinic class}} of a saddle $P$ is the closure of
the transverse intersections of the stable and unstable manifolds
$W^\st(P,f)$ and $W^\ut(P,f)$ of the orbit of $P$. We  denote
this class by $H(P,f)$. A homoclinic class is {\emph{non-trivial}}
if it contains at least two different orbits.

A homoclinic class can be also defined as the closure of the set
of saddles that are homoclinically related with $P$. Here we say that a saddle $Q$ is \emph{homoclinically related} with $P$ if the invariant manifolds of the orbits of $P$ and $Q$ meet cyclically and transversely, that is, $W^\st(P,f) \pitchfork W^\ut(Q,f) \ne
\emptyset$ and $W^\st(Q,f) \pitchfork W^\ut(P,f) \ne \emptyset$.
\label{d.homocliniclass}
\end{definition}


The following is a consequence of our results (see Theorems~\ref{t.complexornontwisted} and \ref{th.proposition} below).

\begin{Theor}\label{t.homoclinic}
Let $f$ be a $C^1$-diffeomorphism with a coindex one cycle
associated to saddles $P$ and $Q$. Suppose that at least one of
the homoclinic classes of these saddles is non-trivial. Then 
the heterodimensional cycle of $f$ associated to $P$ and $Q$ can be 
$C^1$-stabilized.
\end{Theor}

A simple consequence of this result is the following:
 
\begin{Corol}\label{c.corol}
Let $f$ be a $C^1$-diffeomorphism with a heterodimensional cycle
associated to saddles $P$ and $Q$ such that $\sind (P)=\sind(Q)+1$.
Suppose that the intersection 
$W^\ut(P,f)\cap W^\st(Q,f)$ contains at least two different orbits.
Then the cycle can be $C^1$-stabilized.
\end{Corol}

The question of the stabilization of cycles is relevant 
for describing the global dynamics of 
diffeomorphisms (indeed this is another motivation for this paper). 
Let us explain this point succinctly.
Following \cite{Con,Newhouse,A03},
this global  dynamics  is structured
by means of 
homoclinic or/and chain recurrence classes. 
The goal is
to describe the dynamics of these classes and their relating cycles. 
In general,  homoclinic classes 
 are  (properly) contained in chain recurrence classes. 
For  $C^1$-generic diffeomorphisms and  for  hyperbolic periodic points, 
these two kinds of classes coincide,
\cite{BC04}. However,
there are  non-generic situations where  two different homoclinic classes 
are ``joined" by  a cycle. In this case these classes are contained in one common 
chain recurrence class which hence is strictly larger.
We would like to know under which conditions after small perturbations these two homoclinic classes explode and fall into the very same homoclinic class $C^1$-robustly. 
Indeed this  occurs if the cycle can be
$C^1$-stabilized. 
Examples where this stabilization is used 
for describing global dynamics
can be found in \cite{BCDG,Sh,Shbis}.
See \cite[Chapter
10.3-4]{BDV_book} and \cite{bible} for a broader 
discussion of these questions.

To prove our   results we analyze the dynamics associated to
different types of coindex one cycles. This
analysis  essentially depends on two  factors: the
{\emph{central multipliers of the cycle}} and its {\emph{unfolding
map.}} Let us now discuss this point briefly, for further details we
refer to Section~\ref{s.simple}.

\subsection{Multipliers and unfolding map of a cycle}
Let $f$ be a diffeomorphism with a coindex one cycle associated
to saddles $P$ and $Q$. In what follows we will assume that $\sind (P)=\sind (Q)+1$. Denote by $\pi(R)$ the period of a periodic
point $R$.

We say that the cycle is {\emph{partially hyperbolic}} 
if there are heteroclinic points $X\in W^\st(P,f)\cap
W^\ut(Q,f)$ and $Y \in W^\ut(P,f)\cap W^\st(Q,f)$ such that
the closed set formed by the orbits of $P,Q,X,$ and $Y$ has 
a partially hyperbolic splitting of the form $E^\sst\oplus E^c\oplus E^\uut$,
where $E^c$ is one-dimensional, $E^\sst$ is uniformly contracting, and 
$E^\uut$ is uniformly expanding. 
We call $E^\ct$
the {\emph{central bundle}}.
Note that, in particular, this implies that
$X$ is a transverse intersection and $Y$ is a quasi-tranverse intersection
of the invariant manifolds. Also 
 observe that the bundle $E^c$ is necessarily
non-hyperbolic. Bearing in mind this property we
introduce the following definition. 

\begin{definition}[Central multipliers]
The \emph{cycle has real central multipliers} if there are a
contracting real eigenvalue $\la$ of $Df^{\pi(P)}(P)$ and an
expanding real eigenvalue $\be$ of $Df^{\pi(Q)}(Q)$ such that:
\textbf{(i)} $\la$ and $\beta$ have multiplicity one,
\textbf{(ii)} $|\la|> |\sigma|$ for every contracting eigenvalue
$\sigma$ of $Df^{\pi(P)}(P)$, and \textbf{(iii)} $|\beta|< |\eta|$
for every expanding eigenvalue $\eta$ of $Df^{\pi(Q)}(Q)$. In this
case, we say that $\lambda$ and $\beta$ are the {\emph{real
central multipliers of the cycle.}}

Similarly, the \emph{cycle has  non-real central multipliers} if
either \textbf{(i)} there are a pair of  non-real (conjugate)
contracting
 eigenvalues $\la$ and $\bar \la$ of $Df^{\pi(P)}(P)$
such that $|\la|=|\bar \la| \ge |\sigma|$ for every contracting
eigenvalue $\sigma$ of $Df^{\pi(P)}(P)$, or
 \textbf{(ii)} there are a pair of  non-real
(conjugate) expanding
 eigenvalues $\be$ and $\bar \be$ of $Df^{\pi(Q)}(Q)$
such that $|\be|=|\bar \be| \le |\eta|$ for every expanding
eigenvalue $\eta$ of $Df^{\pi(Q)}(Q)$.
\end{definition}

Let us note that cycles with central real multipliers
can be perturbed to get  partially hyperbolic ones (associated to the continuations
of the saddles in the initial one).

In the case of cycles with real central  multipliers we will distinguish
so-called {\emph{twisted}} and {\emph{non-twisted}} cycles, see
Definition~\ref{d.twistenontwisted}. An intuitive explanation of
these two sorts of cycles goes as follows, see Figure~\ref{f.twisted}.

In order to study the dynamics of the 
 cycle we select heteroclinic points $X\in W^\st(P,f)\cap
W^\ut(Q,f)$ and $Y \in W^\ut(P,f)\cap W^\st(Q,f)$. Typically,  $X$
is a transverse intersection point and  $Y$ is a quasi-transverse
intersection point (due to dimension deficiency). 
The next step is to consider a {\emph{neighborhood of the cycle,}}
that is, an open set $V$ containing the orbits of $P,Q,X$, and
$Y$, and study the dynamics of perturbations of $f$ in such a
neighborhood. 
If the neighborhood $V$ is small enough, possibly after a
perturbation of $f$, the dynamics of $f$ in $V$ is
partially hyperbolic 
with a splitting of the form
$E^{\sst}\oplus E^\ct\oplus E^\uut$ (recall the definition above).

Replacing $Y$ by some backward iterate, we can
assume that the heteroclinic point $Y$ is close to $P$. We pick some large number $k$ such
that $f^k(Y)$ is nearby $Q$ and consider the map $\mfT_1=f^{k}$
defined in a small neighborhood of $Y$. This map is called the 
{\emph{unfolding map}}.
 If it is possible to pick $k$ in such a
way that $Df^k$ preserves the orientation of the central bundle then we
say that the cycle is {\emph{non-twisted}}. Otherwise, the
cycle is {\emph{twisted}}. Note
that in the previous discussion the choice of the heteroclinic
point $X$ does not play any relevant role.


\begin{figure}[h]
\centering \scalebox{0.75}{\includegraphics[clip]{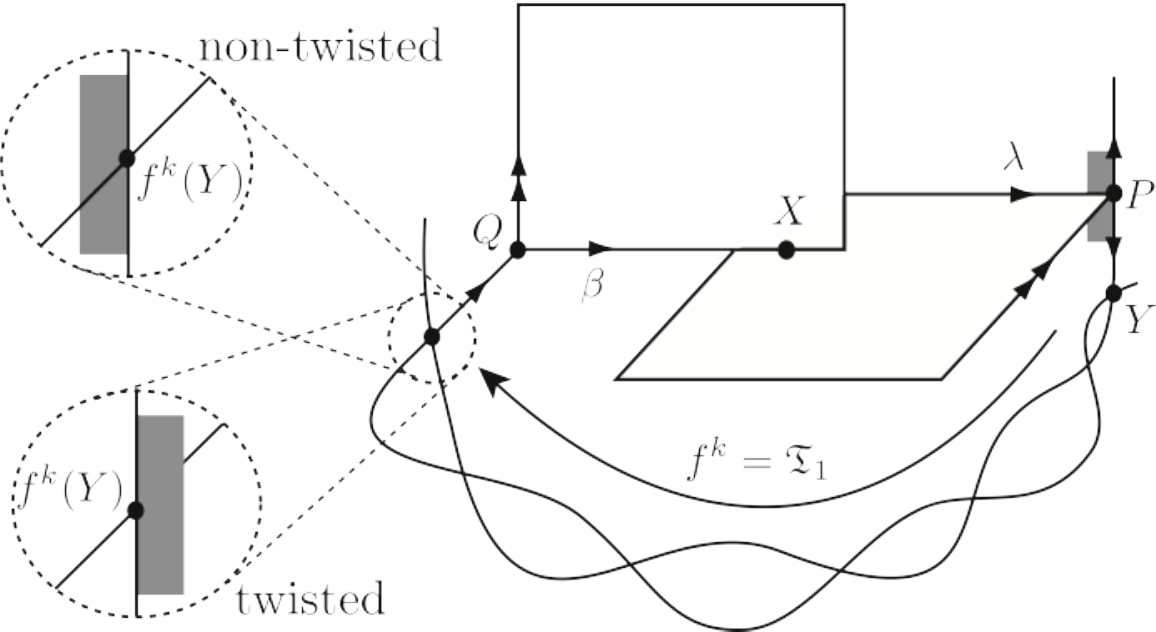}}
\caption{Twisted and non-twisted cycles} \label{f.twisted}
\end{figure}


More precisely, the dynamics of the unfolding of the cycle mostly
depends on the signs of the central eigenvalues $\lambda$
(associated to $P$) and $\beta$ (associated to $Q$) and on the
restriction of $\mfT_1$ to the central bundle. We associate to the
cycle the  \emph{signs}, $\sign(Q)$, $\sign(P)$, and
$\sign(\mfT_{1})$
 in $\{+,-\}$ determined by the
following rules:
    \begin{itemize}
    \item $\sign(Q)=+$ if $\beta>0$ and $\sign(Q)=-$ if $\beta<0$;
    \item  $\sign(P)=+$ if $\lambda>0$ and  $\sign(P)=-$ if $\lambda<0$; and
    \item $\sign(\mfT_{1})=+$  if  $\mfT_1$
    preserves the orientation in the central direction and $\sign(\mfT_{1})=-$  if the orientation is reversed.
    \end{itemize}
    A cycle is twisted if $\sign(Q)=+,$ $\sign(P)=+$, and $\sign (\mfT_1)=-$.
   Otherwise the cycle is non-twisted. 
For details see Definition~\ref{d.twistenontwisted}.


%
%


Let us observe that the discussion above is reminiscent of the
one in \cite[Section 2]{PTannals} about bifurcations of
homoclinic tangencies of surface diffeomorphisms.
It involves similar ingredients to the ones above:  
 the
signs of the eigenvalues of the derivatives, the sides of the
tangencies, and the connections (homoclinic and heteroclinic
intersections).

We are now ready to state our main results. 

\begin{Theor}
\label{t.complexornontwisted} Consider a diffeomorphism $f$ having
a coindex one cycle associated to saddles $P$ and $Q$. Suppose
that
\begin{enumerate}[{\bf(A)}]
\item
either the cycle has a non-real central multiplier,
\item
or the cycle has real multipliers and is non-twisted.
\end{enumerate}
 Then 
the  cycle of $f$ associated to $P$ and $Q$ can be 
$C^1$-stabilized.
\end{Theor}


 Let us observe that
Theorem~\ref{t.complexornontwisted} cannot be improved. Indeed, there are
examples of   diffeomorphisms with twisted cycles that cannot be
stabilized, see \cite{BDcontra}.
On the other hand, we prove that cycles with 
the {\emph{bi-accumulation property}} can be $C^1$-stabilized. 
 Let us state this result more 
precisely.

Given a periodic point $R$ of  $f$, consider
the eigenvalues $\la_1(R),\dots, \la_n(R)$ of $Df^{\pi(R)}(R)$
ordered in increasing modulus  and counted with multiplicity.
If $R$ is hyperbolic, has $\st$-index $k$, and $|\la_{k-1}(R)|<|\la_{k}(R)|$
then there is
a unique invariant manifold $W^\sst (R,f)$ (the strong stable manifold of $R$) 
tangent to the eigenspace associated
to  $\la_1(R),\dots,\la_{k-1}(R)$ (the strong stable bundle). The manifold 
$W^\sst_\loc(R,f)$ has codimension one in  $W^\st_\loc(R,f)$ and
$W^\sst_\loc (R,f)$ splits each component of $W^\st_\loc(R,f)$
into two parts.

\begin{definition}[Bi-accumulation property]
A saddle  $R$ of $\st$-index $k$ such that $|\la_{k-1}(R)|<|\la_k(R)|$
is  {\emph{$\st$-bi-accumulated}} (by homoclinic points) 
if every component of
$W^\st_\loc(R,f)\setminus W^\sst_\loc (R,f)$ contains transverse
homoclinic points of $R$. 

A heterodimensional cycle associated to saddles $P$ and $Q$ 
with $\sind(P)=\sind(Q)+1$ 
 is {\emph{bi-accumulated}}
is either $P$ is $\st$-bi-accumulated for $f$ or 
$Q$ is $\st$-bi-accumulated for $f^{-1}$. 
\label{d.n.biaccumulation}
\end{definition}

In the next result we consider cycles with real central multipliers.

  \begin{Theor}\label{th.proposition}
 $\,$
 \begin{enumerate}[{\bf(A)}]
 \item
 Every
 non-twisted cycle can be $C^1$-stabilized.
\item 
Every twisted cycle with the bi-accumulation property can be 
$C^1$-stabilized.
 \end{enumerate}
 \end{Theor}
 
 Indeed, Theorems~\ref{t.homoclinic} and \ref{t.complexornontwisted}  are consequence of  Theorem~\ref{th.proposition}.

Finally, our results can be summarized as follows:

\begin{Corol}\label{c.summary}
Consider a diffeomorphism $f$ with a fragile cycle associated to saddles $P$ and $Q$ with $\sind (P)=\sind (Q)+1$. Then 
\begin{itemize}
\item
the cycle has positive central real multipliers,
\item
the cycle is persistently twisted (i.e., the cycle cannot be perturbed
to get a non-twisted cycle associated to $P$ and $Q$),
\item
the intersection $W^u(P,f)\cap W^\ut(Q,f)$ consists of exactly one orbit, and
\item
the homoclinic classes of $P$ and $Q$ are both trivial.
\end{itemize}
\end{Corol}
Examples of fragile cycles satisfying the four properties in the corollary
can be found in \cite{BDcontra}.
 
 \section{Ingredients of the proofs}
In this section we review some tools  
of our constructions.

\subsection{Reduction to the case of cycles with real multipliers}
\label{ss.reductiona}

A first step is to see that to prove our results it is enough to consider
cycles with real central multipliers. 
For that let us
recall a result from \cite{BDijmj}.

\begin{theorem}[Theorem 2.1 in \cite{BDijmj}]
Let $f$ be a diffeomorphism with a coindex one cycle associated
to saddles $P$ and $Q$. Then there are diffeomorphisms $g$
arbitrarily $C^1$ close to $f$ with a coindex one cycle with real
central multipliers associated to saddles $P_g^\prime$ and
$Q_g^\prime$ which are homoclinically related to the continuations
$P_g$ and $Q_g$ of $P$ and $Q$. In this result one may have $P=P_g$ and/or $Q=Q_g$.
\label{t.complex}
\end{theorem}

Note that the previous theorem means the following.
\begin{remark}\label{r.non-trivial}
 Assume that the saddle
$P$ in Theorem~\ref{t.complex} has non-real central multipliers. 
Then the homoclinic class of  $P_g^\prime$ 
is non-trivial and contains $P$.
\end{remark}

There is also the following simple fact:

\begin{lemma}\label{l.homoclinicstabilization}
Consider a diffeomorphism $f$ with a heterodimensional cycle associated to
$P$ and $Q$. Suppose that there are saddles 
$P_g^\prime$ and $Q_g^\prime$ homoclinically related to $P_g$ and $Q_g$,
respectively,
with a heterodimensional cycle that can be $C^1$-stabilized. Then the initial
cycle can also be $C^1$-stabilized.
\end{lemma}

\begin{proof}
The stabilization of the cycle associated to $P_g'$ and $Q_g'$ means that
there is $h$ arbitrarily close to $g$ having a pair of basic
hyperbolic sets $\La'_h\ni P'_h$ and $\Si'_h\ni Q'_h$ 
with a robust cycle. Since the saddles $P_h$ and $P_h'$ are homoclinically
related there is a basic set $\La_h$ containing $\La'_h$ and $P_h$.
Similarly, there is a basic set  $\Si_h$ containing $\Si'_h$ and $Q_h$.
Since $W^{i} (\La_h,h) \supset W^{i}(\La_h',h)$ and  
$W^{i} (\Si_h,h) \supset W^{i}(\Si_h',h)$, $i=s,u$,
it is immediate that there is a robust cycle associated to $\La_h\ni P_h$
and $\Si_h\ni Q_h$.
\end{proof}

\begin{remark}\label{r.enough}
Theorem~\ref{t.complex} and Lemma~\ref{l.homoclinicstabilization}
mean that to prove Theorems~\ref{t.homoclinic} and \ref{t.complexornontwisted}
it is enough 
to stabilize cycles with real central multipliers (indeed this is the sort of
cycles considered in Theorem~\ref{th.proposition}).
 Thus  in what follows we will focus on this type cycles.
\end{remark}

\subsection{Strong homoclinic intersections and blenders} 
A key ingredient for obtaining robust
heterodimensional cycles in \cite{BDijmj} is the notion of a
{\emph{blender.}} 
A blender is a hyperbolic set with some additional geometrical intersection
properties that guarantee some robust intersections, 
see Section~\ref{sss.cublender} and 
Definition~\ref{d.blender}.  
 The key  step in \cite{BDijmj} to obtain robust cycles is  that coindex one cycles yield periodic points of saddle-node/flip 
type with
{\emph{strong homoclinic intersections}}: the strong stable
manifold of the saddle-node/flip intersects its strong unstable manifold,
see Definition~\ref{d.stronghomoclinic}. These strong homoclinic
intersections generate blenders yielding robust cycles, see
Proposition~\ref{p.BDblendera}.

In \cite{BDijmj} the generation of blenders is not
controlled and in general the saddle-node/flip has ``nothing to do"
with the saddles in the initial cycle. 
This is why in \cite{BDijmj} the hyperbolic
sets with robust cycles are not related (in general) to the saddles
in the initial cycle. Here
we control the ``generation" of the saddle-node/flip with
strong homoclinic intersections, obtaining blenders that contains the
continuation of
 a saddle in the initial cycle and intersecting the
invariant manifolds of the other saddle in the cycle. 
This configuration 
provides robust cycles associated to hyperbolic sets containing
the continuation of both initial saddles, see 
Theorem~\ref{t.p.BDblenderb}. 

We next explain 
the ``generation" of  saddle-node/flip poits with strong homoclinic intersections.

\subsection{Simple cycles and iterated function systems (IFSs).}
To analyze the dynamics of cycles with real multipliers we
borrow some constructions
and the notion of a
{\emph{simple cycle}}
 from \cite{BDijmj}, see Section~\ref{s.simple}.

In very rough terms, if a diffeomorphism has a simple cycle then
its dynamics 
in a neighborhood of the cycle is affine and preserves a partially
hyperbolic splitting $E^{\sst}\oplus E^\ct\oplus E^\uut$, where
$E^\sst$ is uniformly contracting, $E^\uut$ is uniformly
expanding, and $E^\ct$ is one-dimensional and non-hyperbolic, see
Proposition~\ref{p.simple}.
Following \cite{BDijmj}, to prove our 
results it is enough to consider 
simple cycles and their (suitable) unfoldings. 

 We consider
one-parameter families of diffeomorphisms $(f_t)_t$ unfolding a
simple cycle at $t=0$ and preserving the affine
structure associated to the splitting $E^\sst\oplus E^\ct \oplus
E^\uut$. In particular, the foliation of hyperplanes parallel to
$E^{\sst}\oplus E^\uut$ is preserved. Considering the
{\emph{central dynamics}} given by the quotient of the dynamics of
the diffeomorphism
$f_t$ by these hyperplanes one gets a one-parameter family of
iterated function systems (IFSs).
Some properties
of these IFSs are translated to properties of the
diffeomorphisms $f_t$, see Proposition~\ref{p.dictionary}. 
This IFS
provides
relevant information about the dynamics of the the diffeomorphisms $f_t$ such as, for example, the existence of saddle-nodes with strong homoclinic intersections.
Such IFSs play a role 
similar to the one of the quadratic family in the setting of
homoclinic bifurcations, compare \cite[Chapter 6.3]{PT_book}.

\subsection{Organization of the paper}\label{ss.organization}
The discussion above corresponds to 
the contents in
Sections~\ref{s.steps} and \ref{s.simple}.
The key step is to
analyze the
dynamics of the IFSs associated to simple cycles. 
Using these IFSs,
in
Section~\ref{s.notonesided} we analyze non-twisted cycles (which
is the principal case) and explain how they 
yield 
saddle-nodes/flips with strong homoclinic intersections as
well as further intersection properties, see Proposition~\ref{p.l.R}.
We  study (twisted and non-twisted) cycles with the
bi-accumulation property in Section~\ref{ss.cyclesbiaccumulatedbis}.
In Section~\ref{s.stabilization} we prove Theorem~\ref{th.proposition},
which is the main technical step in the paper.
Finally, in Section~\ref{s.proofoftheoremsAB}
we see how 
Theorems~\ref{t.homoclinic}
and \ref{t.complexornontwisted} 
 can be easily
derived from Theorem~\ref{th.proposition}.




\section{Robust cycles and
blenders}\label{s.steps}

In this section, we recall the definition and main properties of
blenders. We also  state the tools to get the
stabilization of heterodimensional cycles, see
Proposition~\ref{p.BDblendera} and
Theorem~\ref{t.p.BDblenderb}.

\subsection{Blenders} \label{sss.cublender} Let us recall
the definition of a $\cut$-blender in \cite{BDtan}. See also the
examples  in \cite{BD96} and the discussion in \cite[Chapter
6]{BDV_book}:

\begin{definition}[$cu$-blender, Definition 3.1 in \cite{BDtan}]
\label{d.blender} Let $f\colon M\to M$ be a diffeomorphism. A
transitive  hyperbolic compact set $\Ga$ of $f$ with $\uind
(\Ga)=k$, $k\ge 2$, is a {\emph{$\cut$-blender}} if there are a
$C^1$-neighborhood $\cU$ of $f$ and a $C^1$-open set $\cD$ of
embeddings of $(k-1)$-dimensional disks $D$ into $M$ such that
 for every
$g \in \cU$ and
every disk $D\in \cD$ the local stable manifold
$W^s_{\loc}(\Ga_g)$ of $\Ga_g$ 
intersects $\cD$.
The set $\cD$ is called the {\emph{superposition}}
region of the blender.
\end{definition}

\begin{remark} \label{r.blendercontinuation}
Let $\Ga$ be a blender of $f$. 
Then for every $g$ close enough to $f$ the continuation
$\Ga_g$ of $\Ga$ is a blender of $g$.
\end{remark}

In fact, the $\cut$-blenders considered in \cite{BDijmj}
to obtain robust cycles are a
special class of blenders, called {\emph{blender-horseshoes,}} see
\cite[Definition 3.8]{BDtan}. In this definition, the
blender-horseshoe $\Ga$ is the maximal invariant set in a ``cube" $C$
and has a hyperbolic splitting with three non-trivial
bundles $T_\Ga M= E^{\st}\oplus E^{\cut}\oplus E^{\uut}$, such
that the unstable bundle of $\Ga$ is $E^{\ut}=E^{\cut}\oplus
E^{\uut}$ and $E^{\cut}$ is one-dimensional.
Moreover, the set $\Ga$
is conjugate to the complete shift of two symbols. Thus it has
exactly two fixed points, say $A$ and $B$,
called
{\emph{distinguished points of the blender,}} and that 
 play a special role in the definition of a blender-horseshoe.

The definition of a blender-horseshoe involves a $Df$-invariant
strong unstable cone-field $\cC^{\uut}$ corresponding to the
strong unstable direction $E^{\uut}$, the local stable manifolds
$W^{\st}_\loc(A,f)$ and $W^\st_\loc(B,f)$ of the
distinguished saddles $A$ and $B$  (defined as the connected
component of $W^{\st}(R,f)\cap C$ containing $R$, $R=A,\,B$), and
the local strong unstable manifolds $W^{\uut}_\loc(A,f)$ and
$W^\uut_\loc(B,f)$ of  $A$  (the
component of $W^{\uut}(R,f)\cap C$ containing $R$). Recall that
the strong unstable manifold of $R$ is the only invariant manifold
of dimension $\dim (E^{\uut})$ that is tangent to $E^{\uut}$ at
$R$.

Let $\dim (E^{\uut})=u$. 
One considers {\emph{vertical disks}} through the blender, that
is, disks $\De$ of dimension $u$ 
tangent to the cone-field
$\cC^{\uut}$ joining the ``top" and the ``bottom" of the cube $C$.
Then there are two isotopy classes of vertical disks that do not
intersect $W^{\st}_\loc(A,f)$ (resp. $W^\st_\loc(B,f)$), called
disks at the right and at the left of $W^{\st}_\loc(A,f)$ (resp.
$W^\st_\loc(B,f)$). For instance,  $W^{\uut}_\loc(B,f)$ (that is
a vertical disk) is at the right of $W^{\st}_\loc(A,f)$. Similarly,
 $W^{\uut}_\loc(A,f)$ is at the left of
$W^{\st}_\loc(B,f)$. The superposition region $\cD$ of the
blender-horseshoe consists of the vertical disks in between
$W^{\st}_\loc(A,f)$ and $W^{\st}_\loc(B,f)$ (i.e., at the right of
$W^{\st}_\loc(A,f)$ and at the left of $W^{\st}_\loc(B,f)$). See
Figure~\ref{f.blender}.

\begin{figure}[hbtp]
\centering \scalebox{0.75}{\includegraphics[clip]{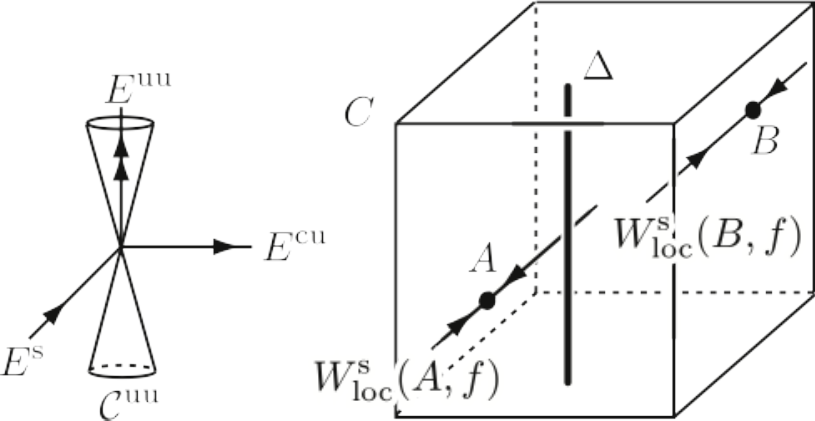}}
\caption{Vertical disks in a blender.} \label{f.blender}
\end{figure}

\subsection{Generation of blenders and robust cycles}\label{sss.generation}

To  state a criterion for the existence of robust
cycles we need some definitions.

\begin{definition}\label{d.stronghomoclinic}
Let $S$ be a periodic point of a diffeomorphism $f$.
\begin{itemize}
\item
We say that $S$ is a {\emph{partially hyperbolic saddle-node
(resp. flip)}} of $f$ if the derivative of $Df^{\pi(S)}(S)$ has
exactly one eigenvalue $\sigma$ of modulus $1$, the eigenvalue $\sigma$  
is equal to  $1$
(resp., $-1$), and there are eigenvalues $\la$ and $\beta$ of
$Df^{\pi(S)}(S)$ with $|\la|<1<|\beta|$.
\item
Consider the strong unstable (resp. stable) invariant direction
$E^{\uut}$ (resp. $E^{\sst}$) corresponding to the eigenvalues
$\kappa$ of $Df^{\pi(S)}(S)$ with $|\kappa|>1$ (resp.
$|\kappa|<1$). The {\emph{strong unstable manifold}}
$W^{\uut}(S,f)$ of $S$ is the unique $f$-invariant manifold
tangent to $E^{\uut}$ of the same dimension as $E^{\uut}$. The
{\emph{strong stable manifold}} $W^{\sst}(S,f)$ of $S$ is defined
similarly considering $E^{\sst}$.
\item
 We say that $S$ has a
{\emph{strong homoclinic intersection}} if $W^{\sst}(S,f)\cap
W^{\uut}(S,f)$ contains points which do not belong to the orbit of
$S$.
\end{itemize}
\end{definition}

\begin{proposition}[Criterion for robust cycles. Theorem 2.4 in \cite{BDijmj}]
Let $f$ be a diffeomorphism having a partially hyperbolic
saddle-node/flip $S$ with a strong homoclinic intersection. Then
there is a diffeomorphism $h$ arbitrarily $C^1$-close to $f$ with
a robust heterodimensional cycle. \label{p.BDblendera}
\end{proposition}

Note that this result does not provide information about the sets
involved in the robust cycle. We state in Theorem~\ref{t.p.BDblenderb}
a version of this proposition providing some information about these
sets. Before proving this theorem 
let us explain the main steps of the proof of Proposition~\ref{p.BDblendera},
for further details see \cite{BDijmj}.

\smallskip

\noindent {\em Sketch of the proof of
Proposition~\ref{p.BDblendera}.} For simplicity, let us assume that
$S$ is a saddle-node of $f$ of period one.
 After a perturbation, we can suppose
that the saddle-node $S$ splits into two hyperbolic fixed points
$S^-_g$ (contracting in the central direction) and $S^+_g$
(expanding in the central direction), here $g$ is a
diffeomorphism obtained by a small the perturbation of $f$. The saddles $S^+_g$
and $S^-_g$ have different indices and the  manifolds
$W^\st(S^-_g)$ and $W^\ut(S^+_g)$ have a transverse intersection
that contains the interior of a ``central" curve joining $S^-_g$ and
$S^+_g$. Note that this intersection property is $C^1$-robust.
 The proof has three steps (see Figure~\ref{f.strongintersection}):
\begin{enumerate}
\item[{\bf (I)}]
There is a blender-horseshoe $\Ga_g$ having $S^+_g$ as a
distinguished  fixed point.
\item[{\bf (II)}]
The unstable manifold of $S^{-}_g$ contains a vertical disk
$\Delta$ in the superposition region $\cD$ of the
blender-horseshoe $\Ga_g$. Thus, by the definition of 
blender-horseshoe,
$W^{\st}(\Ga_g,g)$ intersects $W^\ut(S^-_g,g)$. Hence, as
$S^+_g\in \Ga_g$ and $W^\ut(S^-_g,g)\pitchfork W^\st(S^+_g,g)\ne
\emptyset$, there is a heterodimensional cycle associated to
$\Ga_g$ and $S^-_g$.
\item[{\bf (III)}]
The following properties are open ones:
 {\bf i)} the continuation of the hyperbolic set $\Ga_g$ to be a blender
(the elements in the definition of a blender depend
continuously on $g$, see  Remark~\ref{r.blendercontinuation}), {\bf ii)} $W^\ut(S^-_g,g)$ to contain a
vertical disk in the superposition region  $\cD$ of the blender,
and {\bf iii)} $W^\st(S^-_g,g)\pitchfork W^\ut(S^+_g,g) \ne
\emptyset$.
\end{enumerate}

\begin{figure}[htb]
\centering \scalebox{0.75}{\includegraphics[clip]{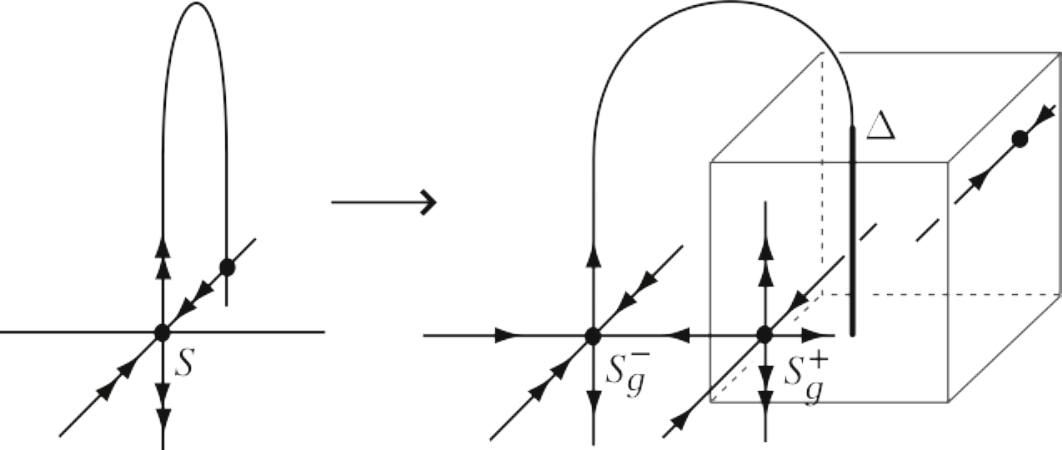}}
\caption{Proof of Proposition~\ref{p.BDblendera}.}
\label{f.strongintersection}
\end{figure}

Therefore, every diffeomorphism $h$  that is $C^1$-close to $g$
has a heterodimensional cycle associated to $S^-_h$ and
$\Gamma_h$. Since $g$ can be taken arbitrarily close to $f$ this concludes
the proof.\hfill \qed

\medskip

Next result is just a  reformulation of the
construction above that allows us to get robust cycles associated
to  sets that contain the continuations of a given saddle. 
This theorem will be the
main tool for stabilizing cycles.

\begin{theorem}
Let $f$ be a diffeomorphism, $P$ a saddle of $f$,  and
 $S$ a partially hyperbolic saddle-node/flip of $f$ such that:
\begin{enumerate}
\item
$ \sind (P)=\dim (W^{\sst}(S))+1=s+1$,
\item
$S$ has a strong homoclinic intersection,
\item
$W^{\ut}(P,f)\cap W^{\sst}(S,f)\ne \emptyset,$ and
\item
$W^{\st}(P,f)\pitchfork W^{\uut}(S,f)\ne \emptyset.$
\end{enumerate}
Then there is a  diffeomorphism $h$ arbitrarily $C^1$-close to $f$
with a robust heterodimensional cycle associated to the
continuation $P_h$ of $P$ and a transitive 
hyperbolic set $\Ga_h$ containing a hyperbolic continuation
$S_h^+$ of $S$ of $\st$-index $s$. \label{t.p.BDblenderb}
\end{theorem}

\begin{proof}
One proceeds as in the proof of Proposition~\ref{p.BDblendera}, 
considering  a perturbation $h$ of
$g$ with saddles $S^\pm_h$ satisfying conditions (I) and (II) above and such that
$$
W^\ut(P_h,h)\pitchfork W^{\st}(S^-_h,h)\ne \emptyset.
$$
Since
$W^\ut(S^-_h,h)\pitchfork W^\st(S^+_h,h)\ne
\emptyset$,
the inclination lemma now implies that
$$
 W^\st(P_h,h)\pitchfork W^{\ut}(S^+_h,h)\ne
\emptyset,
$$
see Figure~\ref{f.P}.

\begin{figure}[hbtp]
\centering \scalebox{0.70}{\includegraphics[clip]{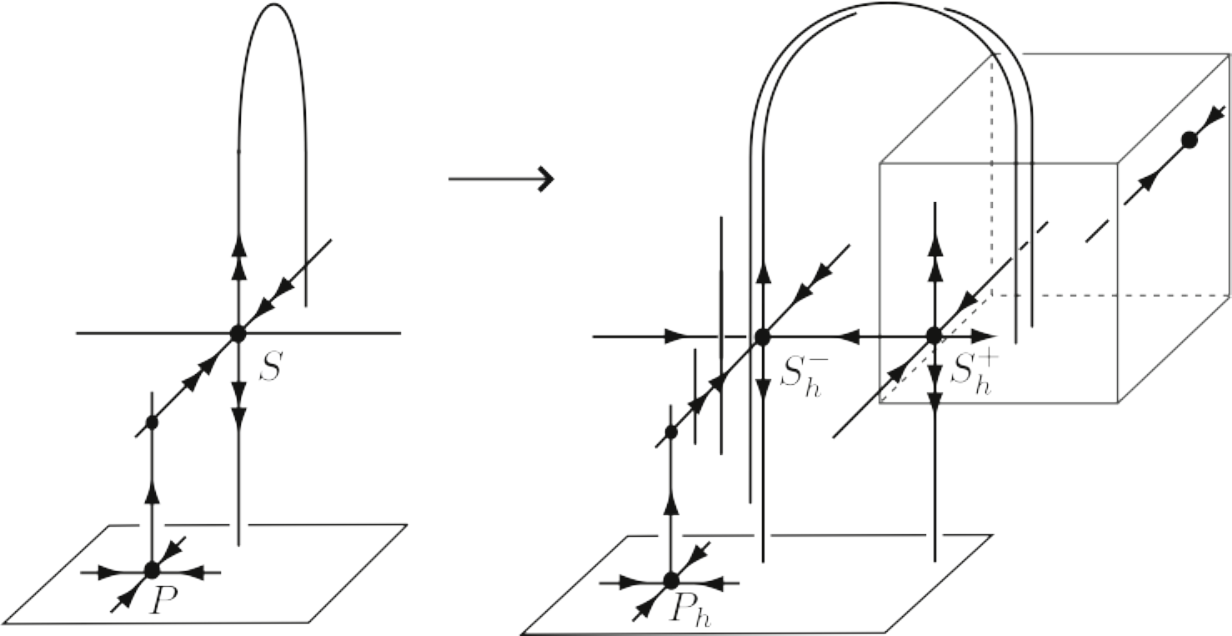}}
\caption{Proof of Theorem~\ref{t.p.BDblenderb}.} \label{f.P}
\end{figure}

Recall that $W^{\ut}(S^-_h,h)$ contains a vertical disk in the
superposition region of the blender $\Ga_h$. 
Since $W^\ut(P_h,h)\pitchfork W^{\st}(S^-_h,h)\ne \emptyset$,
the inclination
lemma implies that the same holds for $W^{\ut}(P_h,h)$. Thus we can repeat the
construction in Proposition~\ref{p.BDblendera} replacing $S^-_h$
by $P_h$. Hence $W^\ut(P_\varphi,\varphi)$ intersects
$W^\st(\Ga_\varphi,\varphi)$ for any diffeomorphism $\varphi$
close to $h$. Since $W^\st(P_\varphi,h)\cap
W^{\ut}(S^+_\varphi,\varphi)\ne \emptyset$ and $S^+_\varphi\in
\Ga_\varphi$ for every $\varphi$ close to $h$, there is a robust heterodimensional cycles associated to
$P_\varphi$ and $\Ga_\varphi$, ending the proof of the
theorem.
\end{proof}




\section{Simple cycles and systems of iterated
funtions}\label{s.simple}

In this section, following \cite{BDijmj}, we introduce simple
cycles (Section~\ref{ss.simple}) and their associated
one-dimensional dynamics (Section~\ref{ss.reduction}). We see that  given any diffeomorphism $f$ with
a co-index one cycle with real
central multipliers (associated to saddles $P$ and $Q$) there is a
diffeomorphism $g$ arbitrarily $C^1$-close to $f$ with a cycle
associated to $P$ and $Q$  whose dynamics in a neighborhood of
the cycle is affine, see Proposition~\ref{p.simple}.
In such a case we say that this cycle of $g$ is simple.

In fact, 
for a diffeomorphism $g$ with a simple cycle
there is a one-parameter family of diffeomorphisms
$(g_t)_t$, $g_0=g$,  preserving a (semi-local) partially
hyperbolic splitting $E^\sst\oplus E^c\oplus E^\uut$ such that the
bundles $E^\sst$ and $E^\uut$ are non-trivial and hyperbolic
(uniformly contracting and uniformly expanding, respectively) and the bundle
 $E^c$ is not hyperbolic and one-dimensional. We consider
the quotient dynamics by the hyperplanes $E^{\sst}\oplus E^\uut$,
obtaining a one-parameter family of one-dimensional
iteration function systems (IFSs) which describe the central
dynamics of the maps $g_t$. Properties of these IFSs are
translated to properties of the
diffeomorphisms $g_t$, see Proposition~\ref{p.dictionary}.

In Section~\ref{s.notonesided} we will  write intersection properties
 implying the existence of robust cycles
 (similar to the ones in Theorem~\ref{t.p.BDblenderb})  
 in terms of properties of the IFSs associated to simple cycles. 
 We now discuss simple cycles and their  IFSs.

\subsection{Simple cycles}\label{ss.simple}

Next proposition summarizes the results in \cite{BDijmj} about
{\emph{simple cycles}} and their {\emph{unfoldings.}}  This proposition means
that if $(f_t)$ is a ``model arc" unfolding a simple cycle
then the dynamics of
the maps $f_t$ in a neighborhood of the cycle is given by 
suitable  compositions of two
linear
maps (the dynamics nearby the saddles in the cycle) and two affine maps (iterations
corresponding to the ``transition" and the ``unfolding maps"). 

\begin{proposition}[Proposition 3.5 and Section~3.2 in
\cite{BDijmj}] Let $f$ be a  diffeomorphism having a co-index one
cycle with real central multipliers  associated to saddles $P$ and
$Q$ such that
$$
\sind (Q)+1=\sind (P).
$$ 
Then there is a
one-parameter family of diffeomorphisms $(f_t)_{t\in [-\epsilon,
\epsilon]}$, $\epsilon>0$, such that it
satisfies properties (C1)--(C3) below
and
 $f_0$ is arbitrarily close to
$f$.

 Let $s$ and $u$ be the
dimensions of $W^\st(Q,f)$ and of $W^\ut(P,f)$, respectively. There
are linear maps
\begin{itemize}
\item
$\phi_\la,\psi_\be \colon \RR\to \RR$, $\phi_\la(x)=\la\,x$ and
$\psi_\beta(x)=\beta\, (x)$,
\item
$A^\st,B^\st,T_1^\st,T_2^\st\colon \RR^s\to \RR^s$, which are
contractions (i.e., their norms are strictly less than one),
\item
$A^\ut,B^\ut,T_1^\ut,T_2^\ut\colon \RR^u\to \RR^u$, which are
expansions (i.e., their inverse maps are contractions),
\end{itemize}
such that:

\medskip

\noindent{{\bf (C1)}} There are local charts $U_P$ and $U_Q$
centered at $P$ and $Q$ such that in
these coordinates we have, for  all $t$,
$$
\begin{array}{ll}
f_t^{\pi(P)}(x^s,x^c,x^u)&=(A^\st(x^s),\phi_\lambda(x^c),
A^\ut(x^u)),
\\
f_t^{\pi(Q)} (x^s,x^c,x^u)&=(B^\st(x^s),\psi_\beta(x^c),
B^\ut(x^u)),
\end{array}
$$
where $|\la|\in (0,1)$ and $|\be|>1$, $x^s\in \RR^s$, $x^c\in
\RR$, and $x^u\in \RR^u$, and $\pi(P)$ and $\pi(Q)$ are the periods of $P$ and $Q$, respectively.

\medskip

\noindent {{\bf (C2)}}
 There is a quasi-transverse
heteroclinic point $Y_P\in W^{\st}(Q,f_0)\cap W^\ut(P,f_0)$ in
$U_P$
such that, in  the coordinates in the chart $U_P$, it holds:
\begin{enumerate}
\item\label{i.qtP}
For every $t$,
 $Y_P=(0^s,0,a^u)\in W^\ut_\loc(P,f_t)$,  $a^u\in\RR^u$.
\item
There is a neighborhood $C^\st(Y_P)$ of $Y_P$  in
$W^\st(Q,f_0)\cap U_P$  of the form
$(-1,1)^{s}\times\{(0,a^{u})\}$.
\item\label{i.qtQ}
There is $\tpq\in \NN$ such that for all $t$
$$
Y_{Q,t}=(a^s,t,0^u)= f_t^{\tpq} (Y_P)\in U_Q\cap W^\ut(P,f_t),
\quad a^s\in\RR^s,
$$
and
$$
Y_{Q,t}\in C^\ut(Y_{Q,t})=\{(a^s,t)\}\times (-1,1)^{u} \subset
W^\ut(P,f_t)\cap U_Q.
$$
\item
There is a neighborhood $U_{Y_P}$ of $Y_P$, $U_{Y_P}\subset U_P$,
such that
$$
\mfT_{1,t}=f_t^\tpq \colon U_{Y_P}\to f^\tpq(U_{Y_P}) \subset U_Q
$$
is an affine map of the form
\[
\begin{split}
 \mfT_{1,t}(x^s,x^c,x^u) &=
\mfT_{1}(x^s,x^c,x^u)+(0,t,0)\\
&= \big(T_1^\st(x^s), \pm x^c, T^\ut_1(x^u)\big)+
\big(a^s,t,-T_1^\ut (a^u)\big)\\
&=\big(T_1^\st(x^s)+a^s_t, \theta_{1,t}( x^c),
T^\ut_1(x^u)-T_1^\ut (a^u)\big).
\end{split}
\]
\end{enumerate}

\begin{figure}[hbtp]
\centering \scalebox{0.75}{\includegraphics[clip]{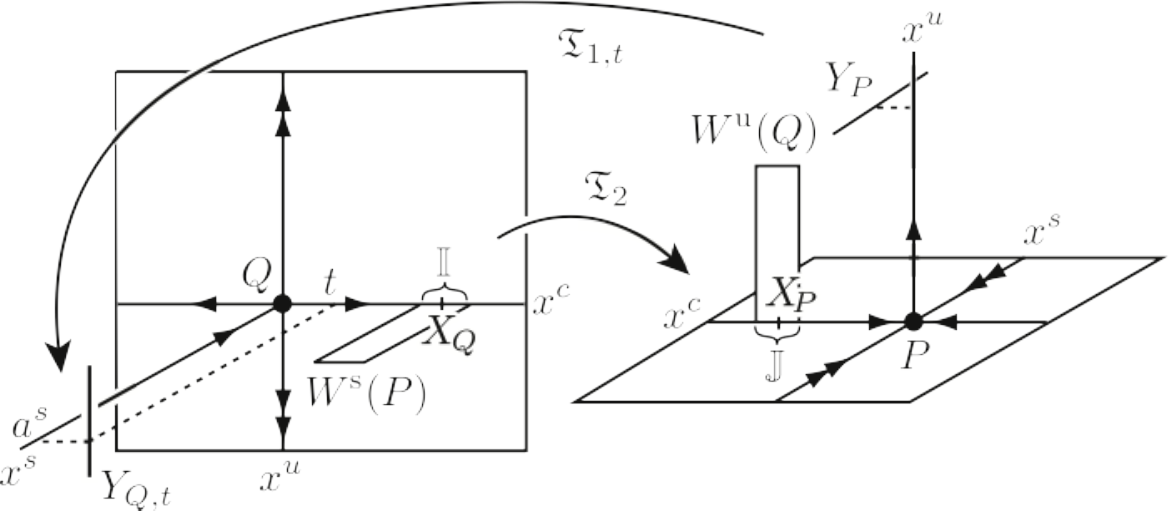}}
\caption{} \label{f.simple}
\end{figure}

\noindent {{\bf (C3)}} For every $t$, there is a point $X_Q\in
U_Q$  in $W^{\ut}(Q,f_t)\pitchfork
W^\st(P,f_t)$  (independent of $t$)
such that, in the  coordinates in the chart
$U_Q$, it holds:
\begin{enumerate}
\item
\label{i.intervalI} $X_Q=(0^s,1,0^u)$ and there is $\delta>0$ such
that
$$
X_Q\subset \II=\{0^s\}\times [1-\delta,1+\delta]\times \{0^u\}
\subset W^{\ut}(Q,f_t)\pitchfork W^\st(P,f_t).
$$
\item
There is $\tqp\in \NN$ such that $X_P=f^\tqp_t (X_Q)=(0,-1,0)\in
U_P$ and
$$
X_P\in  \JJ=f^\tqp_t(\II)=\{0^s\}\times [-1-\delta,-1+\delta]\times
\{0^u\}\subset U_P.
$$
\item
There is a neighborhood $U_{X_Q}$ of $X_Q$,  $U_{X_Q}\subset U_Q$,
such that
$$
\mfT_{2,t}=\mfT_2=f^\tqp_t \colon U_{X_Q}\to f^\tqp(U_{X_Q})
\subset U_P
 $$
is an affine map of the form
\[
\begin{split}
 \mfT_{2}(x^s,x^c,x^u)&=\big(T_2^\st(x^s),
\pm (x^c-1), T_2^\ut (x^u)\big) + (0^s,-1,0^u)\\
&=\big(T_2^\st(x^s), \theta_2 (x^c), T_2^\ut (x^u)\big).
\end{split}
\]
\end{enumerate}
\label{p.simple}
\end{proposition}

According to \cite[Sections
3.1-2]{BDijmj}, we give the following definition.

\begin{definition}[Simple cycles]
\label{d.simplecycles} The map $f_0$ in Proposition~\ref{p.simple}
has  a {\emph{simple cycle}} and $(f_t)_{t\in [-\epsilon,
\epsilon]} $ is a {\emph{model unfolding family}} of $f_0$.
\begin{itemize}
\item
$\mfT_{1,t}$ and $\mfT_{2}$ are the {\emph{unfolding}} and the
{\emph{transition maps}},  
\item
$\theta_{1,t}$ and $\theta_2$ are
the {\emph{central unfolding}} and the {\emph{central transition
maps,}}
\item $\tpq$ and $\tqp$ are
the {\emph{unfolding}} and the {\emph{transition times,}}
\item
$\lambda$ and $\beta$ are the {\emph{central multipliers,}} 
and
\item
$\phi_\la(x)=\la\,x$ and $\psi_\beta(x)=\be\,x$ are the
{\emph{linear central maps}} of the cycle.
\end{itemize}
\end{definition}

\begin{remark}
Since we are only interested in the dynamics in the central
direction of the simple cycle, we denote the simple cycle and its
unfolding model
 by $\scyc (f,Q,P,\be, \la,\pm_1, \pm_2)$, where the
 symbols $\pm_1$ and $\pm_2$ refer to the orientation preservation
 or reversion of the maps $\mfT_{1}$ and $\mfT_2$, respectively.
These symbols coincide with the choices of $\pm$ in (C2)(4) and (C3)(3).
To emphasize the unfolding and the transition times
$\tpq$ and $\tqp$ we will write $\scyc (f,Q,P,\be,
\la,\pm_{1}, \pm_2, \tpq,\tqp)$.
\end{remark}

We now state some generalizations of the
simple cycles above.

\subsubsection{Simple cycles with homoclinic intersections and 
semi-simple cycles}
\label{ss.simpleshomoclinic}
In our constructions we will  consider cycles associated
to saddles with  non-trivial homoclinic classes. We want that some of
these homoclinic intersections associated to this saddle were
``detected" by the cycle and ``well posed" in relation to it.
This leads to the next definition.

\begin{definition}[Simple cycles with adapted homoclinic
intersections] Consider a simple cycle 
$\scyc (f,Q,P,\be, \la,\pm_{1},
\pm_2)$. 
Write  $f=f_0$ and let $(f_t)_{t\in
[-\ve,\ve]}$ be a model unfolding family of $f_0$. 
The family
$(f_t)_{t\in [-\ve,\ve]}$ has {\em{adapted homoclinic
intersections}} (associated to $P$) if it satisfies conditions
(C1)--(C3) in Proposition~\ref{p.simple} and

\medskip

\noindent {{\bf (C4)}} In the local coordinates in $U_Q$, there is
$\bar a^s\in (-1,1)^s$ such that
$$
\Delta_0=\{(\bar a^s,1)\}\times [-1,1]^u\subset W^\ut(P,f_t), \quad
\mbox{for every  $t$ close to $0$.}
$$
This implies that
$(\bar a^s,1,0)$ is a transverse homoclinic point of $P$ of
$f_t$ for all $t$ close to $0$.

\medskip


The  family
$(f_t)_{t\in [-\ve,\ve]}$  has a {\emph{sequence of
adapted homoclinic intersections}} (associated to $P$) if it
satisfies conditions (C1)--(C4) and 

\medskip

\noindent {{\bf (C5)}} In the local coordinates in $U_Q$, 
for every $t$ close to $0$
there
are sequences 
$$
\bar a_i^s\to \bar a^s \quad \mbox{and} \quad 
x_i\to 1, \qquad
\bar
a_i^s\in (-1,1)^s \quad \mbox{and} \quad x_i \in (1-\delta, 1+\delta),
$$ such
that 
$$
\De_i=\{(\bar a_i^s,x_i)\}\times
[-1,1]^u\subset W^\ut(P,f_t)\quad \mbox{for every  $t$ close to $0$.}
$$  
Moreover, the orbits by $f_t$ of 
the disks $\De_i$, $i\ge 0$, are pairwise disjoint. 

As above, this implies that
$(\bar a_i^s,x_i,0)$ is a
transverse homoclinic point of $P$ of $f_t$.

\medskip

In these cases, we say that $f_0$ has a {\em{simple cycle with an
adapted (sequence of) homoclinic intersection(s).\/}}
\label{d.adaptedsimple}
\end{definition}

Since we will  consider perturbations of 
simple cycles, in some cases we will need to consider diffeomorphisms
with ``simple cycles" such that the maps $\psi_\be$ and $\phi_\la$
in Proposition~\ref{p.simple} are not linear.

\begin{definition}[Semi-simple cycles]
\label{d.nonaffinesimple}  A diffeomorphism $f$ has a
{\emph{semi-simple cycle}} associated to saddles $P$ and $Q$
if it satisfies Proposition~\ref{p.simple} where the linear central maps
$\phi_\la$ and $\psi_\be$ in (C1) are replaced by maps $ \tilde
\phi_\la, \tilde\psi_\beta\colon \RR \to \RR $ with
$$
\tilde \phi_\la(0)=\tilde \psi_\beta(0)=0, \quad \tilde
\phi_\la^\prime(0)=\la, \quad \tilde \psi_\be^\prime(0)=\beta.
$$
For such a semi-simple cycle we use the notation $\sscyc
(f,Q,P,\tilde\psi_\be, \tilde\phi_\la,\pm_{1}, \pm_2)$.
\end{definition}

\subsection{Twisted and non-twisted cycles}\label{ss.twistednontwisted}
To a simple cycle $\scyc (f,Q,P,\be, \la,\pm_{1},
\pm_2)$ we associate signs $\sign(Q)$, $\sign(P)$, and
$\sign(\mfT_{1})$
 in $\{+,-\}$ by the
following rules:
  \begin{itemize}
    \item $\sign(Q)=+$ if $\beta>0$ and $\sign(Q)=-$ if $\beta<0$,
    \item  $\sign(P)=+$ if $\lambda>0$ and  $\sign(P)=-$ if $\lambda<0$, and
    \item $\sign(\mfT_{1})=+$  if  $\pm_1=+$ (i.e., $\theta_{1,0}(x^c)=x^c$)
    and $\sign(\mfT_{1})=-$  if $\pm_1=-$  (i.e., $\theta_{1,0}(x^c)=-x^c$).
    \end{itemize}

\begin{definition}[Twisted and non-twisted cycles]
\label{d.twistenontwisted} We say that a simple cycle \newline
$\scyc
(f,Q,P,\be, \la,\pm_{1}, \pm_2)$ is {\emph{twisted}} if
$(\sign(Q),\sign(P),\sign(\mfT_{1}))= (+,+,-)$. Otherwise the
cycle is {\emph{non-twisted.}}

 A diffeomorphism $f$ with a co-index one
cycle with real central multipliers (associated to $P$ and $Q$) is
{\emph{twisted}} (resp. {\emph{non-twisted}}) if there is a
diffeomorphism $h$ arbitrarily $C^1$-close to $f$ with a twisted (resp.
non-twisted) simple cycle associated to $P$ and $Q$.
\end{definition}


Next lemma means that after a perturbation
non-twisted  cycles can be chosen satisfying $(\text{sign}(Q),\text{sign}(P),
\text{sign}(\mfT_{1}))=(\pm,\pm,+)$ (i.e., the case $(-,-,-)$ can be
discarded).

\begin{lemma}\label{l.reduction}
Consider a non-twisted simple cycle
 $\scyc (f,Q,P,\be, \la,\pm_{1},
\pm_2)$. 
Then there is a
diffeomorphism $g$ arbitrarily close to $f$ with a simple cycle
associated to $P$ and $Q$ of such that
$$
(\sign (Q),\sign (P), \sign (\mfT^g_{1}))=(\pm,\pm,+).
$$
This notation  emphasizes that  $\mfT^g_{1}$ is
the unfolding map of the  cycle associated to $g$.
\end{lemma}


\begin{proof}
If $ \sign (\mfT_{1})=+$ we are done. If  $\sign
(\mfT_{1})=-$ then the definition of non-twisted cycle implies that
at least one of the central multipliers $\la$ and $\be$ of
 the cycle is negative. To prove the lemma we fix  
 a constant $K>0$ (with
$K>|\be|^2$ and $K^{-1}<|\la|^2$) and replace 
the unfolding map
$\mfT_{1,0}$ by a
composition of the form
$$
(Df^{\pi(Q)})^m \circ \mfT_{1,0}\circ (Df^{\pi(P)})^n,
$$
where $n$ and $m$ are arbitrarily large and
$$
\la^n\,\beta^m<0 \quad \mbox{and} \quad K^{-1}<|\la^n\,\beta^m|<K.
$$
In this way, we get a new ``unfolding map" $\bar
\mfT_{1,0}=f^m\circ \mfT_{1,0}\circ f^n$, defined on a small
neighborhood of $f^{-n}(Y_P)$, where $Y_P\in
W^\st(Q,f)\cap W^\ut(P,f)$   is the heteroclinic
point in (C2) in
Proposition~\ref{p.simple}.
By construction, the  central component $\bar
\theta_{1,0}$ of $\bar
\mfT_{1,0}$
satisfies
$$
\bar \theta_{1,0}(x^c)=-\la^n\,\beta^m \,
x^c=|\la^n\,\beta^m|\, x^c.
$$
Consider now the segment of orbit
$$
\{ f^{-n}(Y_P),\dots , Y_P, \dots , f^{\tpq}(Y_P),\dots,
f^{\tpq+m}(Y_P)\}.
$$
 Since $n$ and $m$ are arbitrarily big and $K^{-1}<|\la^n\,\beta^m|<K$,
we can modify the map $f$ along this segment of orbit
 to get $\bar
\theta_{1,0}(x^c)=x^c$. This perturbation can be taken arbitrarily
small if $n$ and $m$ are arbitrarily large. Therefore the new
simple cycle is of type $(\pm,\pm,+)$.
This completes the sketch of the proof of the lemma.
 For further details see
\cite[Proposition 3.5]{BDijmj}.
\end{proof}

\subsection{Quotient dynamics. Families of iterated function
systems}\label{ss.reduction} In what follows, $(f_t)_{t\in
[-\epsilon,\epsilon]}$ is a model unfolding family associated to a
diffeomorphism $f=f_0$ with a semi-simple cycle. 
We use the notation in Proposition~\ref{p.simple}.
Next remark allows us to consider (in a neighborhood of a semi-simple
cycle) the quotient dynamics by the strong stable/unstable
hyperplanes.

\begin{rem}\label{r.preserve}
{\em{Consider a semi-simple cycle
$\sscyc
(f,Q,P,\tilde\psi_\be, \tilde\phi_\la,\pm_{1}, \pm_2)$ and its
mo\-del unfolding map  $(f_t)_{t\in
[-\ve,\ve]}$, where $f_0=f$. Consider the
partially hyperbolic splitting $E^{\sst}\oplus E^\ct\oplus
E^{\uut}$, defined over the orbits of $P$ and $Q$, that in the
local charts $U_P$ and $U_Q$ is of the form
$$
E^{\sst}=\RR^s\times\{(0,0^u)\}, \quad E^\ct=\{0^s\}\times
\RR\times\{0^u\}, \quad E^{\uut}=\{(0^s,0)\}\times \RR^u.
$$
This splitting  is extended to $U_P\cup U_Q$ as constant bundles.
Proposition~\ref{p.simple} implies that the maps $\mfT_{1,t}$ and
$\mfT_{2}$ are affine maps preserving $E^{\sst}\oplus E^\ct\oplus
E^{\uut}$.

The open set $V$ defined by
\begin{equation}
\label{e.v} V=U_P\cup U_Q \cup  \left( \bigcup_{i=0}^{\tqp}
f_0^i(U_{X_Q}) \right) \cup \left( \bigcup_{i=0}^{\tpq}
f_0^{i}(U_{Y_P}) \right)
\end{equation}
is the {\emph{neighborhood associated to the  cycle.}} For
small $t$, we consider  the maximal invariant set $\La_t(V)$  of $f_t$ in $V$,
$$
\La_t (V)= \bigcap_{i\in \ZZ} f_t^i(V).
$$
By construction, for $f_t$ there is a partially hyperbolic extension of the 
splitting $E^{\sst}\oplus E^\ct\oplus
E^{\uut}$
 over the set $\La_t(V)$. With a slight abuse of notation,
we also denote this extension by $E^{\sst}\oplus E^\ct\oplus
E^{\uut}$.
}}
\end{rem}

This remark implies that the returns of
points  $X\in U_{X_Q} \cap \La_t(V)$ to $U_{X_Q}$, 
$$
X\in U_{X_Q}
\cap \La_t(V) \mapsto f_t^i(X)\in U_{X_Q},
$$
 preserve the
codimension one foliation $\RR^s\times \{x^c\} \times \RR^u$
tangent to $E^{\sst}\oplus E^{\uut}$. We consider the ``quotient
dynamics" by these hyperplanes, obtaining a one parameter family
 of  iterated function systems (IFS) defined
on the interval $\II=[1-\de,1+\de]$ (see item (\ref{i.intervalI})
in (C3) in Proposition~\ref{p.simple}). This family 
describes the ``central" dynamics of these returns. 
We will provide in Proposition~\ref{p.dictionary} a 
``dictionary" translating
properties of this IFS to properties of the diffeomorphisms
$f_t$. These properties
are about the existence of periodic orbits, homoclinic and
heteroclinic intersections, and cycles.


\subsubsection{Families of IFSs induced by the quotient dynamics}\label{sss.quotient}
Consider a semi-simple cycle $\sscyc (f,Q,P,\psi_\be, \phi_\la,\pm_1,
\pm_2)$ and its model unfolding family $(f_t)_{t\in [-\epsilon,
\epsilon]}$, here $f=f_0$.
Consider the segment $\II$ in  condition  (C3)(1) in
Proposition~\ref{p.simple}.
For each pair $(k,n)$ of large natural numbers and small $t$,
define the map
\begin{equation}\label{IFS}
\Ga_{t}^{k,n} \colon \II_{t}^{k,n} \to \II, \quad \Ga_{t}^{k,n}(x)=
(\psi_\beta^k\circ \theta_{1,t}\circ \phi_\lambda^n \circ
\theta_2)(x),
\end{equation}
where $\II_{t}^{k,n}$ is the maximal subinterval of $\II$ where
the map $\Ga_{t}^{k,n}$  is defined. Note that there are choices of
$k,n,t$ such that the set $\II_t^{k,n}$ is empty.

The one-parameter family  $(\Ga_{t}^{k,n})_{t\in
[-\epsilon,\epsilon]}$
 is {\emph{the IFS associated to $(f_t)_{t\in [-\epsilon,\epsilon]}$}.

%

\subsubsection{Dictionary IFS -- Global dynamics}
\label{sss.dictionary} 
Using the invariance of the  spitting $E^\sst\oplus E^c\oplus E^\uut$ above
one gets the following 
 extension  of \cite[Proposition 3.8]{BDijmj}:

\begin{proposition}[Quotient dynamics -- Global dynamics]
Consider a semi-simple cycle $\sscyc (f,Q,P,\psi_\be, \phi_\la,\pm_{1},
\pm_2,\tpq,\tqp)$, its  model unfolding family $(f_t)_{t\in
[-\ve,\ve]}$, here $f=f_0$, and its associated IFS
$(\Ga_t^{n,m})_{t\in [-\epsilon,\epsilon]}$. Suppose that the
saddles $P$ and $Q$ have $\st$-indices $(s+1)$ and $s$,
respectively.

\medskip

\noindent {\bf (A) Periodic points:} Suppose that there is $r\in
\II_{t}^{k,n}$ such that
$$
\Ga_{t}^{k,n} (r)=r.
$$
Then there are $r^s\in \RR^s$ and $r^u\in \RR^u$ 
such that 
$$
R=(r^s,r,r^u)\in U_Q\cap \La_t(V)
$$
is a periodic point of
$f_t$
 of period 
$$
\pi(R)=k\, \pi(Q)+n\, \pi(P)+\tpq+\tqp.
$$
The
 eigenvalue of $Df^{\pi(R)}_t(R)$ 
corresponding to central direction $\{0^{s}\}\times \RR \times
\{0^u\}$ is
$$
\Big( \Ga_{t}^{k,n} \Big)^\prime (r)= \Big( \psi^k_\beta
\Big)^\prime \big( \theta_{1,t}(  \phi^n_\lambda (\theta_2
(r)))\big)\, \Big( \phi^n_\lambda
\Big)^\prime \big(\theta_2 (r)\big).
$$
In particular, if  $\big|\big(\Ga_{t}^{k,n}\big)^\prime (r)\big|
>1$ (resp. $<1$) the periodic point $R$ has $\st$-index $s$ (resp.
$\st$-index $s+1$).

Moreover, the periodic point $R$ also satisfies
\begin{equation}\label{e.initemA}
W^{\sst}(R, f_t)\pitchfork W^\ut(Q, f_t)\ne\emptyset \quad
\mbox{and} \quad W^{\uut}(R, f_t)\pitchfork W^\st(P,
f_t)\ne\emptyset.
\end{equation}

\medskip

In what follows, let $r$, $R$, and $(k,n)$ be as in item (A).

\medskip

\noindent {\bf (B)  Strong homoclinic intersections:} Suppose that
there is a pair $(\bar k,\bar n)\ne (k,n)$ such that
$$
 \Ga_{t}^{\bar k,\bar n} (r)=r.
$$
Then $W^{\sst}(R, f_t)\cap W^{\uut}(R, f_t)$
contains points that do not belong to the orbit of $R$.

\medskip

\noindent {\bf (C)  Heterodimensional cycles:} Suppose that there
are $d\in \II$ and $d^s\in\RR^s$ such that (in the coordinates in
$U_Q$)
$$
\Upsilon = \Upsilon (d^s,d)= \{(d^s,d)\} \times [-1,1]^u \subset W^\ut(P, f_t).
$$
If there is $i\in \NN$ such that
$$
\theta_{1,t} \circ \phi_\lambda^i \circ \theta_2 (d)=0
$$
then 
$$
W^{\ut}(P,f_t)\cap W^{\st}(Q,f_t)\ne \emptyset.
$$
Thus,
as
$W^{\st}(P,f_t)\cap W^{\ut}(Q,f_t)\ne \emptyset$, the
diffeomorphism $f_t$ has a heterodimensional cycle associated to
$P$ and $Q$.

In particular, if there are $i,h\in\NN$ such that
$$
\theta_{1,t} \circ  \phi_\lambda^i \circ \theta_2 \circ
\psi_\beta^h (t)= \theta_{1,t} \circ  \phi_\lambda^i \circ
\theta_2 \circ  \psi_\beta^h \circ \theta_{1,t}(0)=0
$$
then $f_t$ has a heterodimensional cycle associated to
$P$ and $Q$\footnote{In the previous expression one
implicitly assumes that $\psi_\be^h(t)\in [1-\de,1+\de]$, otherwise one
cannot apply $\theta_2$.}.

\medskip

\noindent {\bf (D)  Heteroclinic intersections (I):}  Suppose that
there are $i,\tilde k, \tilde n \in \NN$ such that
$$
\theta_{1,t} \circ  \phi_\lambda^i \circ \theta_2\circ
\Ga_{t}^{\tilde k,\tilde n} (r)=0.
$$
Then 
$$
W^{\uut}(R, f_t)\cap W^{\st}(Q,f_t)\ne \emptyset.
$$
If $(\tilde k, \tilde n)=(0,0)$ the previous identity just means $
\theta_{1,t} \circ  \phi_\lambda^i \circ \theta_2(r)=0. $

\medskip

\noindent {\bf (E)  Heteroclinic intersections (II):} Let
$(d^s,d)$ be as in item (C) (i.e., $\Upsilon(d^s,d)\subset W^\ut(P,f_t)$).
If there are $i,j\in \NN$ such that
$$
\Ga_t^{i,j}(d)=r
$$
then 
$$
W^{\ut}(P,f_t)\cap W^{\sst}(R,f_t)\ne 0.
$$
In particular, if 
\begin{enumerate}
\item \label{i.item1Edictionary}
either $r=d$ and $(i,j)=(0,0)$,
\item \label{i.item2Edictionary}
or there is $i$ such that $
 \psi_\beta^i \circ \theta_{1,t}(0)=  \psi_\beta^i(t)=r$
\end{enumerate}
then 
$$
W^{\ut}(P,f_t)\cap
W^{\sst}(R,f_t) \ne \emptyset.
$$

\medskip

\noindent {\bf (F)  Homoclinic points:} Suppose that there is $i$
such that
$$
\psi_\beta^i \circ \theta_{1,t}(0)= \psi_\beta^i(t)=\hat h\in [1-\de,1+\de].
$$
Then there is $\hat h^s\in (-1,1)^s$
such that $\hat H=(\hat h^s,\hat h,0^u)\in U_Q$ is a transverse
homoclinic point of $P$ for $f_t$ and
$$
\{(\hat h^s,\hat h)\}\times [-1,1]^u \subset W^\ut(P,f_t).
$$
 \label{p.dictionary}
\end{proposition}

\begin{proof}
For notational  simplicity, let us assume that
$P$ and $Q$ are fixed points.

Items (A) and (B) are stated in \cite[Proposition 3.8]{BDijmj}. To
prove item (A) it is enough to observe that the definition of the
pair $(k,n)$ and the product structure provide a pair of cubes
$\De^u\subset [-1,1]^u$ and $\De^s\subset [-1,1]^s$ such that
$$
f^\ell_t\big( [-1,1]^s\times \{r \} \times \De^u \big)=
\De^s\times \{r\} \times [-1,1]^u, \quad \ell = k+n+\tpq+\tqp,
$$
if $k$ and $n$ are large enough (note that $k,n\to \infty$ as $t\to 0$).
Note that $D f_t^\ell$  uniformly contracts vectors
parallel to $\RR^s\times \{(0,0^u)\}$ and uniformly expands
vectors parallel to $\{(0^s,0)\}\times \RR^u$. This gives the
periodic point $R=(r^s,r,r^u)$ of period $\ell$. Note
that our arguments also imply  that
\begin{equation}\label{e.uussR}
W^{\uut}(R,f_t)\supset \{(r^s,r)\} \times [-1,1]^u, \quad
W^{\sst}(R,f_t)\supset [-1,1]^s \times \{(r,r^u)\}.
\end{equation}
Note also that from (C3)(1) in Proposition~\ref{p.simple}, 
 in the
coordinates in $U_Q$, one has that
\begin{equation}
\label{e.rssuu}
\begin{split}
&\{0^s\}\times  [1-\delta,1+\delta] \times [-1,1]^u \subset
W^\ut(Q, f_t), \quad \mbox{and} \\
& [-1,1]^s\times [1-\delta,1+\delta] \times \{0^u\} \subset
W^\st(P,f_t).
\end{split}
\end{equation}
The intersection properties between the invariant manifolds of
$R,P,$ and $Q$ in item (A) follow immediately from equations
\eqref{e.uussR} and \eqref{e.rssuu} and $r\in [1-\de,1+\de]$.

\smallskip

To prove item (B) one argues exactly as in item (A). Note that
the choice of $(\bar k,\bar n)$ (large $\bar k, \bar n$) provides  a cube
$\tilde\De^u\subset [-1,1]^u$ and a point $\tilde r^s\in [-1,1]^s$
such that
$$
f^m_t\big( \{(r^s, r) \} \times \tilde \De^u \big)= 
\{(\tilde r^s,r)\} \times [-1,1]^u, \quad m = \bar k +\bar
n+\tpq+\tqp.
$$
Since $\{(r^s, r )\} \times \tilde \De^u\subset W^\uut (R,f_t)$
and $(\tilde r^s,r,r^u)\in W^\sst(R,f_t)$ there is a strong
homoclinic intersection associated to $R$.

\smallskip

 To prove the first part
of item (C) note that if $t$ is small then $i$ is large and
thus
\[
\begin{split}
f_t^{\tpq+i+\tqp}(\Upsilon)\cap U_Q&=
f_t^{\tpq+i+\tqp}\big(\{(d^s,d)\} \times
[-1,1]^u \big) \cap U_Q \\
&\supset \{(\bar d^s, \theta_{1,t}\circ  \phi_\la^i \circ \theta_2
(d))\} \times [-1,1]^u\\&= \{(\bar d^s,0)\} \times [-1,1]^u,
\end{split}
\]
for some $\bar d^s\in (-1,1)^s$. Since $[-1,1]^s\times
\{(0,0^u)\}\subset W^\st (Q,f_t)$ and $\Upsilon\subset
W^\ut(P,f_t)$ we get $W^\ut (P,f_t) \cap W^\st (Q,f_t)\ne
\emptyset$.

To prove the second part of item (C) consider $a^s\in \RR^s$ and
the linear map $B^s$ as in (C2)(3) and  (C1) in
Proposition~\ref{p.simple}, respectively. Note that
$$ \big( (B^s)^h(a^s), \psi^h_\beta (t), 0^u \big)=
(\tilde d^s, \tilde d,0^u ),
\quad \tilde d=\psi^h_\beta (t)=
\psi^h_\beta \circ \theta_{1,t}(0)
\in [1-\de,1+\de]
$$
is a transverse homoclinic point of $P$ such that
$$
\tilde \Upsilon= \{ (\tilde d^s, \tilde d)\}\times [-1,1]^u
\subset W^\ut(P,f_t) \cap U_Q.
$$
The intersection between $W^\ut(P,f_t)$ and $W^\st(Q,f_t)$
now follows applying the first part of item (C) to the
disk $\tilde \Upsilon$: just note that by hypothesis
and the definition of $\tilde d=\psi^h_\beta \circ \theta_{1,t}(0)$ one has 
 $\theta_{1,t}
\circ \phi_\lambda^i \circ \theta_2 (\tilde d)=0$.

\smallskip

Item (D) follows similarly. Let (in the coordinates in $U_Q$)
$$
 \De =\{(r^s,r)\}\times [-1,1]^u \subset W^{\uut}(R,f_t).
$$
In the coordinates in $U_Q$, we have
\[
\begin{split}
f_t^{\tpq+i+\tqp+\tilde n + \tpq + \tilde k +\tqp} ( \De) &
\supset \{(\tilde r^s, \theta_{1,t} \circ  \phi_\lambda^i \circ
\theta_2 \circ  \Ga_t^{\tilde n,\tilde
k} (r))\}\times [-1,1]^u \\
& = \{(\tilde r^s,0)\}\times [-1,1]^u,
\end{split}
\]
for some $\tilde r^s$. As $[-1,1]^s\times \{(0,0^u)\} \subset
W^\st (Q,f_t)$ we get  $W^\uut(R,f_t)\cap W^\st (Q, f_t)\ne\emptyset$.

\smallskip

The remainder assertions (E) and (F) in the proposition follow
analogously, so we omit their proofs.
\end{proof}




\section{Simple non-twisted cycles}
\label{s.notonesided}


In this section we first consider non-twisted cycles 
and explain how these cycles yield partially hyperbolic
saddle-node/flip points with strong homoclinic intersections as
well as further intersection properties, see Proposition~\ref{p.l.R}.
Using Proposition~\ref{p.dictionary} we will write these 
properties in terms of the  IFSs associated to the cycle.
We also see how  these intersections
are realized by perturbations (model families) of the initial cycle.
These intersection properties are the main ingredient for
the stabilization of cycles. 
Finally, in Section~\ref{ss.cyclesbiaccumulatedbis}
we consider cycles involving a saddle with a non-trivial homoclinic class
and introduce the bi-accumulation property.

\subsection{Non-twisted simple cycles with adapted
homoclinic intersections}

The first step is to see that non-twisted simple cycles yield
simple cycles with adapted homoclinic intersections. 

\begin{lemma}
\label{l.adapted} Consider a non-twisted cycle $\scyc
(f,Q,P,\be,\la,\pm_1,\pm_2)$. There is $g$ arbitrarily
$C^1$-close to $f$ having a non-twisted simple cycle (associated
to $Q$ and $P$) with a sequence of adapted homoclinic intersections (associated
to $P$).
\end{lemma}

\begin{proof}
Note that by Lemma~\ref{l.reduction} we 
can assume that $\theta_{1,t}(x)=x+t$.
The proof has two steps. We first perturb the cycle to get a cycle
with one adapted homoclinic intersection. In the second step we 
perturb this new cycle with an adapted homoclinic intersection to get a cycle
with a sequence of adapted homoclinic intersections.

\smallskip

\noindent{\emph{A cycle with one adapted homoclinic intersection.}}
Observe that, after an arbitrarily small perturbation, we can
assume that the central multipliers of the cycle
satisfy
$\lambda^k=\beta^{-m}>0$ for some arbitrarily large $k$ and $m$. We fix
small $t_k>0$ such that
\begin{equation}\label{e.tk}
t_k=\lambda^k=\beta^{-m}.
\end{equation}
This choice gives
$$
\psi_\beta^m \big( \theta_{1,t_k}(0) \big)=\psi_\beta^m (t_k)=1.
$$
Therefore, by (F) in Proposition~\ref{p.dictionary},  
the point $H=(h^s,1,0)\in U_Q$
 is a
transverse homoclinic point of $P$  such that
$$
\{(h^s,1)\} \times [-1,1]^u \subset W^u(P,f_{t_k}).
$$
The point $H$ will provide the adapted homoclinic point in
Definition~\ref{d.adaptedsimple}.

To see that $f_{t_k}$ has a cycle associated to $P$ and $Q$ just
note that
\begin{equation}\label{e.condition}
\theta_{1,t_k}\circ \phi_\lambda^k \circ \theta_2 \circ
\psi_\beta^m \circ \theta_{1,t_k}(0)= \theta_{1,t_k}\circ
\phi_\lambda^k \circ \theta_2(1)=-\lambda^k+t_k=0.
\end{equation}
 Item (C) in
Proposition~\ref{p.dictionary} implies that $W^\ut(P,f_{t_k})\cap 
W^\st(Q,f_{t_k}) \ne \emptyset$.  

Let $\tilde Y_P$ be the
heteroclinic point in $W^\ut(P,f_{t_k}) \cap W^\st(Q,f_{t_k})$
corresponding to the condition in (\ref{e.condition}).
This implies  that $f_{t_k}$ has a cycle associated to $P$ and
$Q$ and that the points $X_Q\in W^\st(P,f_{t_k}) \cap
W^\ut(Q,f_{t_k})$ (in condition (C3)(1)) and $\tilde Y_P\in
W^\ut(P,f_{t_k}) \cap W^\st(Q,f_{t_k})$ are heteroclinic points
associated to this cycle. Using the transverse homoclinic
 point $H$ of $P$ 
 and arguing as in Lemma~\ref{l.reduction}, 
 we will get
a cycle with an adapted homoclinic intersection.

\smallskip

Indeed, repeating the previous argument we can assume that
the cycle has two ``adapted homoclinic points". The 
additional one is of the form $V=(v^s, 1+v, v^u)$, where
$1+v\in [1-\de, 1+\de]$  (in principle $v\ne 0$) 
and $\De_V=\{(v^s,1+v)\}\times [-1,1]^u\subset W^u(P,f_t)$.
We also have that the disks $\De_V$
and $\De_H=\{(h^s,1)\}\times [-1,1]^u\subset W^u(P,f_t)$
have disjoint orbits.
We use the disk $\De_V$ to get the sequence of adapted homoclinic intersections.

\smallskip

\noindent{\emph{A cycle with a sequence of adapted homoclinic intersections.}}
To get a cycle with a sequence of adapted homoclinic intersections
we argue as above, but now starting with a cycle with ``two adapted homoclinic
intersections", say $H$ and $V$ as above. 
Let us assume that $\theta_2(1+x)=(-1+x)$.
The case $\theta_2(1+x)=(-1-x)$ is analogous.
As above we can assume that equation \eqref{e.tk} holds for infinitely many
$m$ and $k$.

To get a sequence of homoclinic points $H_i$ accumulating to $H$
write 
$$
\de_i=\be^m\,\la^{2\,i}\,(1-v)>0
$$
and consider the sequence
\[
\begin{split}
\psi_\be^m\circ \theta_{1,t_k} \circ \phi_\la^{2\,i}\circ
\theta_2(1+v)& = \psi_\be^m \big( t_k-\la^{2\,i}\, (1- v) \big)\\
&=1-\be^m\,\la^{2\,i}\,(1- v)=1-\delta_i.
\end{split}
\]
Item (F) in Proposition~\ref{p.dictionary} implies that for each
$i$ there is $h^s_i$ such that
\begin{equation}\label{e.Hi}
H_i=(h^s_i,1-\de_i,0)\in U_Q, \quad \de_i>0,
\end{equation}
is a 
transverse homoclinic point of $P$ and 
$$
\De_i=\{(h^s_i,1-\de_i)\} \times
[-1,1]^u\subset W^u(P,f_{t_k}).
$$
This sequence accumulates to $\De_H$ and the disks
$\De_i$ and $\De_H$ have disjoint orbits by construction.

Finally,
arguing exactly as above we have that $f_{t_k}$ has a heterodimensional
cycle associated to $P$ and $Q$.

Write $\tilde f=f_{t_k}$. We perturb $\tilde f$
to get a simple cycle with a sequence of adapted homoclinic intersections.
 Note that $\tilde f$ preserves the
partially hyperbolic splitting $E^{ss}\oplus E^c\oplus E^{uu}$ in
the neighborhood $V$ of the initial simple cycle (recall
(\ref{e.v})). For this new cycle we have ``transition maps''  say
$\tilde \mfT_{1,t_k}$ and $\tilde \mfT_2$ (in principle, these
maps do not satisfy all the properties of  ``true"
transitions). These new ``transitions" $\tilde \mfT_{1,t_k}$  and
$\tilde \mfT_2$ are obtained considering compositions of the maps
$\mfT_{1,t_k}$, $\mfT_2$, $Df^{\pi(P)}(P)$, and $Df^{\pi(Q)}(Q)$
defined for the initial cycle and replacing the heteroclinic
points $X_Q$ and $\tilde Y_P$ by some backward iterates of them.
 Note that the central maps $\tilde \theta_{1,0}$ and $\tilde \theta_2$
 associated to the ``new transitions"
 may fail to
be isometries. 

Now, exactly as in the proof of
Lemma~\ref{l.reduction},  we consider an arbitrarily small
perturbation of $\tilde f$ obtained taking multiplications (in the
central direction) by numbers close to one  throughout long
segments of the orbits of $X_Q$ and $\tilde Y_P$.
This is possible since
$t_k$ can be taken arbitrarily small and $k$ and $m$ arbitrarily big.
 The resulting
diffeomorphism has a simple cycle with a sequence of adapted
homoclinic intersections associated to $P$ (obtained considering
appropriate iterations of the points $H_i$ and $H$). This
completes  the proof of the lemma.
\end{proof}

\begin{remark}\label{r.adaptedintersections}
Using equation \eqref{e.Hi},
we can assume that in the coordinates in  $U_Q$, the adapted
transverse homoclinic points of $P$ are such that
$$
\begin{array}{ll}
H&=(h^s,1,0^u) \quad \mbox{and} \quad \{(h^s,1)\}\times
[-1,1]^u\subset W^u(P,f),\\
H_i&=(h^s_i,\zeta_i,0^u) \quad \mbox{and} \quad
\{(h^s_i,\zeta_i)\}\times [-1,1]^u\subset W^u(P,f),
\end{array}
$$
where $(\zeta_i)$ is an increasing sequence converging to $1$.
\end{remark}

\subsection{Dynamics generated by non-twisted cycles}
\label{ss.perturbationsofnontwisted}
Consider a diffeomorphism $f$ with a simple cycle and its
associated neighborhood $V$ in \eqref{e.v}. For $g$ close to $f$
let $\La_g (V)= \cap_{i\in \ZZ} g^i(V)$ be the maximal invariant set
of $g$ in $V$.
Note that the set $\La_g(V)$ has a partially hyperbolic splitting
of the form 
$E^{\sst}_g\oplus E^\ct_g \oplus E^\uut_g$, where $E^\ct_g$ is one-dimensional
and $E^\sst_g$ and $E^\uut$ uniformly contracting and expanding, respectively.

\begin{proposition}
\label{p.l.R} Consider a non-twisted cycle $\scyc
(f,Q,P,\be,\la,+,\pm_2)$ with a sequence of adapted 
homoclinic intersections (associated to $P$). 
Then there
is a diffeomorphism $g$ arbitrarily $C^1$-close to $f$ 
with a partially hyperbolic
saddle-node/flip $S_g\in \La_g(V)$ of arbitrarily large period 
satisfying the following
properties:
\begin{enumerate}
\item \label{i.R1} $W^{\sst}(S_g,g) \pitchfork W^{\ut}(Q,g)\ne\emptyset$, 
\item  \label{i.R2}  $W^{\uut}(S_g,g) \pitchfork W^{\st}(P,g)\ne\emptyset$,
\item \label{i.R4}
$W^\ut(P,g) \cap W^{\sst}(S_g,g)\ne\emptyset$,
\item \label{i.R5}  $W^{\ut}(P,g) \cap
W^{\st}(Q,g)\ne\emptyset$ and this intersection is quasi-transverse, and
\item\label{i.R6} the homoclinic class
of $P$ for $g$ is non-trivial.
\end{enumerate}
\end{proposition}

\begin{remark}
Indeed, the proof of this proposition
will  imply that the strong unstable manifold of $S_g$ 
transversely intersects the disk $[-1,1]^s\times \II\times \{0^u\}$ 
contained in $W^\st(P,g)$ in
(C3)-\eqref{i.intervalI}  in Proposition~\ref{p.simple}.  Now item
\eqref{i.R4} in Proposition~\ref{p.l.R} implies that
$W^\ut(P,g)$ accumulates to $W^\uut(S_g,g)$ (may be after a perturbation). Thus 
 after a 
perturbation we can assume that 
$W^\ut(P,g)\pitchfork \big([-1,1]^s\times \II\times \{0^u\} \big)\ne \emptyset$.
\label{r.preadapted}
\end{remark}

\subsubsection{Proof of Proposition~\ref{p.l.R}}
The main step in the proof of the proposition 
 is the next lemma about the
IFS associated to a simple cycle.

\begin{lemma}\label{l.ifs}
Consider a non-twisted cycle $\scyc
(f,Q,P,\be,\la,+, \pm_2)$ with an increasing
 sequence of
adapted homoclinic intersections  $(h_i^s,\zeta_i,0^u)$ as in
Remark~\ref{r.adaptedintersections}.

Then there are  sequences of parameters $(t_i)_i$, $t_i\to 0$, and
of perturbations $ \psi_{\be,i}$ of $\psi_\be(x)=\beta\,x$,
$\psi_{\be,i} \to \psi_{\be}$, 
such that the IFS 
$\tilde \Ga_{t_i}^{n,k}$ associated to 
$\phi_\lambda$, $\psi_{\beta,i}$, $\theta_{1,t_i}$, and $\theta_2$ in equation \eqref{IFS}
 satisfies the following properties:
\begin{enumerate}
\item \label{i.ifs1}
There is a sequence of pairs $(v_i,w_i)$,
$v_i,w_i\to \infty$, such that
$$
\begin{array}{c}
\tilde \Gamma_{t_i}^{v_i,w_i}(1)=1,\\
 \dfrac{\la^2}{2\, (1-|\la|^2)}< |(\tilde
\Gamma_{t_i}^{v_i,w_i})^\prime (1))|< \dfrac{2\, |\la |}{1-|\la|}.
\end{array}
$$
\item \label{i.ifs3}
There are large $j$ and $\ell\in \NN$ such that 
$$
\theta_{1,t_i}\circ \phi_\lambda^{\ell} \circ \theta_2
(\zeta_j)=0.
$$ 
\item \label{i.ifs4}
There are $j_0\in \{j-1,j+1\}$
($j$ as in item (\ref{i.ifs3})) and $\bar
n,\bar \ell\in \NN$ such that
$$
\Gamma_{t_i}^{\bar n,\bar \ell} (\zeta_{j_0})=1.
$$
\end{enumerate}
\end{lemma}

We postpone the proof of this lemma to the next subsection. 

\begin{proof}[Proof of Proposition~\ref{p.l.R}]
Note that for each $t_i$  there is a perturbation
$f_i$ of $f$, $f_i\to f$ as $i\to \infty$,  having a semi-simple cycle
$\sscyc (f_i,Q,P, \psi_{\be,i} , \la, +,\pm_2)$
``close" to the initial cycle 
$\scyc
(f,Q,P,\be,\la,+,\pm_2)$ (i.e., we replace the linear map $\psi_\be$
by its perturbation $\psi_{\be,i}$, while preserving the cycle configuration).

For large $i$,
write $g=f_{i}$ and
select the pair
$(v_i,w_i)$ in item (\ref{i.ifs1}) of Lemma~\ref{l.ifs}.
Let
 $S_g=(s^s,1,s^u)$ be the saddle associated to
 this pair and the central coordinate ``$1$" 
  given by
(A) in Proposition~\ref{p.dictionary}. 
By construction, the
eigenvalue $\la_c(S_g)$ of $Dg^{\pi(S_g)}(S_g)$ corresponding to
the central direction $E^\ct_g$ satisfies
\begin{equation*}
 \dfrac{|\la |^2 }{2\, (1-|\la|^2)}< |\la_c(S_g)|<
 \dfrac{2\,|\la|}{1-|\la|}.
\end{equation*}
We claim that $S_g$ also satisfies the intersection properties in the proposition
(note that in principle  $S_g$ is not yet a saddle-node/flip).
\begin{itemize}
\item 
Itens
(\ref{i.R1}) and (\ref{i.R2}) in the proposition  follow from
equation \eqref{e.initemA}
in 
 item (A) of
Proposition~\ref{p.dictionary}.
\item Item (\ref{i.R4}) in the proposition follows from
(\ref{i.ifs4}) in Lemma~\ref{l.ifs} and (E) in
Proposition~\ref{p.dictionary}, where $d=\zeta_{j\pm 1}$
corresponds to  adapted homoclinic points (recall also
Remark~\ref{r.adaptedintersections}). Note that using these points we
also get that $W^\ut(P,g)$ transversely intersects $[-1,1]^s\times \II
\times \{0^u\}$, proving 
Remark~\ref{r.preadapted}.
\item Item (\ref{i.R5}) in the proposition
follows from (\ref{i.ifs3}) in Lemma~\ref{l.ifs}
 and (C) in
Proposition~\ref{p.dictionary},  where $d=\zeta_{j}$ corresponds to
an adapted homoclinic point.
\item Since transverse homoclinic intersections persist and
the saddle $P$ has transverse homoclinic points 
for the diffeomorphism $f$, we get \eqref{i.R6} in the proposition.
\end{itemize}

It remains to see that we can take $S_g$ with
$\la_c(S_g)=\pm 1$.
Observe that the period $\pi(S_g)$ of $S_g$ can be taken arbitrarily large and 
$|\la_c(S_g)|$ is uniformly bounded (independent
of the period). Arguing as in Lemma~\ref{l.reduction}, we
 perturb $g$ along the orbit of
$S_g$ in order to transform this point into a saddle-node (if $\la_c(S_g)>0$)
or a flip (if $\la_c(S_g)<0$). In this way one gets a partially hyperbolic saddle-node/flip.
This perturbation can be done preserving the intersection properties 
in the proposition. 
This concludes the proof of the proposition.
\end{proof}

\subsubsection{Proof of Lemma~\ref{l.ifs}}
Let us first consider the case $\be>0$ and $\la>0$.

\medskip

\noindent {\em{Positive central multipliers:}} 
As above, after an
arbitrarily small perturbation of the central multipliers of
cycle, we can assume that there are arbitrarily large $m$ and $k$
with
\begin{equation}\label{e.bothpositive}
\beta^{-m}=\lambda^{k}\, (1-\lambda).
\end{equation}
Consider the parameter $t_k=\lambda^k$. This choice gives
$$\
\begin{array}{ll}
\Gamma_{t_k}^{m,k+1}(1)&= \psi^m_\be \circ \theta_{1,t_k}\circ
\phi_\lambda^{k+1} \circ \theta_2(1)= \psi^m_\be \circ
\theta_{1,t_k} (-\lambda^{k+1})\\
&= \beta^m\,
(\lambda^k-\lambda^{k+1})=\beta^m\,\lambda^k\,(1-\lambda)=1.
\end{array}
$$
Take $(v_k,w_k)=(m,k+1)$ and note that
$$
(\Gamma_{t_k}^{v_k,w_k})^\prime (1)=\pm \, \be^m\, \la^{k+1}= \pm
\frac{\la}{1-\la}.
$$
This gives (\ref{i.ifs1})  in the lemma. To obtain the other
conditions we consider perturbations $\tilde \psi_\be$ of
$\psi_\be$  preserving the condition $\Gamma_{t_k}^{v_k,w_k}(1)=1$.
 From now on we fix the parameter $t_k$.
We first consider the case where $\theta_2$ has derivative +1.

\smallskip

\noindent {\em{Case $\theta_2(1+x)=-1+x$:}} For every small enough
$\mu$, define $\be(\mu)$ by
$$
\be(\mu)^m\, (\la^k (1-\la) +\mu)=1
$$
and  consider its associated linear map $\psi_{\be(\mu)}(x)=\beta(\mu)\, x$.
 Write $\phi_\la(x)=\la\,x$.
 Note that the 
 IFS $\tilde \Ga_{t_k+\mu}^{i,j}$
 associated to $\phi_\la$, $\psi_{\be(\mu)}$, $\theta_{t_k+\mu}$, and $\theta_2$
  satisfies
\begin{equation}\label{e.fixedpoint}
\tilde \Ga_{t_k+\mu}^{m,k+1} (1)= \psi_{\be(\mu)}^m\circ
\theta_{1,t_k+\mu} \circ \phi_\la^{k+1}\circ \theta_2(1)= 1,
\quad \mbox{for all small $\mu$.}
\end{equation}
Thus, for $(v_k,w_k)=(m,k+1)$,
\begin{equation}
\label{e.recallderivative}
(\tilde \Gamma_{t_k+\mu}^{v_k,w_k})^\prime (1)=
\be(\mu)^m\,\la^{k+1}= \frac{\la}{1-\la+\mu}.
\end{equation}
Thus, for small $\mu$, these derivatives also satisfy
(\ref{i.ifs1}).

Consider $\zeta_i$ as in Remark~\ref{r.adaptedintersections},
that is
$\zeta_i=1-\de_i$, $\de_i\to 0^+$ and $\de_i>\de_{i+1}$. For
large $i$ define
\begin{equation}
\label{e.omegai}
\omega_i(\mu)=\theta_{1,t_k+\mu}\circ \phi_\lambda^k \circ
\theta_2(\zeta_i)=\theta_{1,t_k+\mu}(-\la^k-\la^{k}\, \de_i)= \mu
-\la^{k}\,\de_i.
\end{equation}
Note that
\begin{equation}\label{e.diference}
\omega_{i+1}(\mu)-\omega_i(\mu)= \la^k\, (\de_i-\de_{i+1}).
\end{equation}
Define small $\mu_j>0$ by the condition
$$
\omega_j(\mu_j)=0, \qquad \mu_j=\la^k\,\de_j, \qquad  \lim_{j \to
\infty} \mu_j\to 0.
$$
By the choice of $\mu_j$ and \eqref{e.diference} one has
$$
 \omega_{j+1}(\mu_j)= \la^k\, (\de_j-\de_{j+1}), \qquad
\lim_{j\to \infty} \omega_{j+1}(\mu_j)\to 0^+.
$$
In particular, $\omega_{j+1}(\mu_j)$ can be taken arbitrarily
small in comparison with $\be(\mu_j)^{-m}= \la^k\, (1-\la)+\mu_j$.
This immediately implies the following:

\begin{fact}
Given any $N>0$ there is large $j$ such that
 $[\omega_{j+1}(\mu_j), \be(\mu_j)^{-m}]$ contains
at least $N$ consecutive fundamental domains of 
 $\psi_{\be(\mu_j)}$.
\end{fact}

Using this fact, we get that
for every large $j$ there is a small  perturbation $\tilde \psi_{\be(\mu_j)}$ of
the linear map
$\psi_{\be(\mu_j)}$ such that:
\begin{itemize}
\item 
$\tilde \psi_{\be(\mu_j)}(x)=\psi_{\be(\mu_j)}(x)$ if $x\in
[\be(\mu_j)^{-m-1},1]$.
\item
There is large $n_j$ such that $\tilde \psi_{\be(\mu_j)}^{n_j}
(\omega_{j+1}(\mu_j))=\be(\mu_j)^{-m}$. 
\item
The maps $\tilde \psi_{\be(\mu_j)}$ and 
$\psi_{\be(\mu_j)}$  coincide in a small neighborhood of $0$.
\item
The size of the perturbation goes to $0$ as $j\to \infty$.
\end{itemize}

\begin{remark}
Note that the first two conditions above imply that
\begin{equation}\label{e.summarizing}
\tilde
\psi_{\be(\mu_j)}^{n_j+m} (\omega_{j+1}(\mu_j))=1.
\end{equation}
Also important, note that this perturbation can be done (and we do)
in such a way
previous conditions   \eqref{e.fixedpoint},
 \eqref{e.recallderivative}, and
 \eqref{e.omegai} are preserved.
\end{remark}

The previous construction can be summarized as follows. 
Fix large $k$ and the sequence of  parameters $t_{k,j}=t_k+\mu_j$.  For each large $j$,
consider
 the perturbation
$\tilde \psi_{\be(\mu_j)}$ of $\psi_{\be}$ and the IFS $\tilde
\Ga^{\ell,n}_{t_{k,j}}$ corresponding to $\tilde
\psi_{\be(\mu_j)}$, $\phi_\la$, $\theta_{1,t_{k,j}}$, and $\theta_2$.
Then
\begin{enumerate}[{\bf(i)}]
\item
$\tilde \Gamma_{t_{k,j}}^{v_k,w_k}(1)=1$, (recall \eqref{e.fixedpoint}),
\item
$(\tilde \Gamma_{t_{k,j}}^{v_k,w_k})^\prime (1)=
\dfrac{\la}{1-\la+\mu_j}$, (recall \eqref{e.recallderivative}),
\item $\theta_{1,t_{k,j}}\circ
\phi_\lambda^k \circ \theta_2(\zeta_j)=
\omega_j (\mu_j)=
0$, (recall the choice of $\mu_j$ and \eqref{e.omegai}), and
\item
$\tilde
\Ga_{t_{k,j}}^{n_j+m,k}(\zeta_{j+1})=
 \tilde
\psi_{\beta(\mu_j)}^{n_j+m} \circ \theta_{1,t_{k,j}}\circ
\phi_\lambda^k \circ \theta_2(\zeta_{j+1})=
\tilde
\psi_{\be(\mu_j)}^{n_j+m} (\omega_{j+1}(\mu_j))=
1$, (recall \eqref{e.summarizing}).
\end{enumerate}

To conclude the proof the lemma in this first  case
(positive multipliers and $\theta_2(1+x)=-1+x$)
 just note that
 (i)--(ii) correspond to (\ref{i.ifs1}) in the lemma,
 (iii) to
 (\ref{i.ifs3}) in the lemma, and (iv) to
(\ref{i.ifs4}) in the lemma.

\smallskip

\noindent {\emph{Case $\theta_2(1+x)=-1-x$:}} We proceed as in the
previous case and define the sequence $\omega_i(\mu)$ similarly.
In this case, instead equation \eqref{e.omegai} we get
$$
\omega_i(\mu)=\theta_{1,t_k+\mu}\circ \phi_\lambda^k \circ
\theta_2(\zeta_i)=\theta_{1,t_k+\mu}(-\la^k+\la^{k}\, \de_i)= \mu
+\la^{k}\,\de_i.
$$
We define $\mu_j$ as above, $\omega_j(\mu_j)=0$, and consider
$\omega_{j-1}(\mu_j)>0$ instead of $\omega_{j+1}(\mu_j)$. The
proof now follows as above.

\medskip

\noindent {\emph{Non-positive central multipliers:}} In this case,
after an arbitrarily small perturbation of the central multipliers
of cycle, we can assume that there are arbitrarily large $m$ and
$k$ with
\begin{equation}\label{e.somenonpositive}
\beta^{-2\,m}=\lambda^{2\,k}\, (1-\lambda^2).
\end{equation}
We consider the parameter $t_k=\la^{2\,k}$. The proof now follows
exactly as in the case where the multipliers are both positive
considering the sequences
$$
\omega_i(\mu)=\theta_{1,t_k+\mu}\circ \phi_\lambda^{2\,k} \circ
\theta_2(\zeta_i)=\theta_{1,t_k+\mu}(-\la^{2\,k}\pm \la^{2\,k}\,
\de_i)= \mu \pm \la^{2\,k}\,\de_i.
$$
This completes the proof of Lemma~\ref{l.ifs}. \hfill \qed

\subsection{Cycles associated to a bi-accumulated saddles}
\label{ss.cyclesbiaccumulatedbis}

Given a periodic point $R$ of  $f$, consider
the eigenvalues $\la_1(R),\dots, \la_n(R)$ of $Df^{\pi(R)}(R)$
ordered in increasing modulus  and counted with multiplicity.
Denote by
$\operatorname{Per}^{k} (f)$ the set of (hyperbolic) saddles $R$ of
$f$ of
$\st$-index $k$ satisfying $|\la_{k-1}(R)|<|\la_{k}(R)|<1$.
Given such a saddle $R\in \operatorname{Per}^{k} (f)$, its
local strong
stable manifold $W^\sst_\loc (R,f)$ is well defined 
(recall that $W^\sst (R,f)$ is
the unique invariant manifold tangent to the eigenspace associated
to  $\la_1(R),\dots,\la_{k-1}(R)$). Moreover,
$W^\sst_\loc(R,f)$ has codimension one in  $W^\st_\loc(R,f)$ and
$W^\st_\loc(R,f)
 \setminus W^\sst_\loc (R,f)$ has $2\,\pi(R)$ connected
components (indeed $W^\sst_\loc (R,f)$ splits each component of $W^\st_\loc(R,f)$
into two parts).

 Given a saddle $P$
of $\st$-index $s+1$, we consider the
following subsets of $H(P,f)$:
\begin{itemize}
\item $\operatorname{Per}_h (H(P,f))$ is the subset of $H(P,f)$ of hyperbolic
periodic points $R$ which are homoclinically related to $P$ (thus $R$ also has index $(s+1)$),
\item
$\operatorname{Per}_h^{s+1} (H(P,f))=\operatorname{Per}_h (H(P,f))
\cap \operatorname{Per}^{s+1} (f)$.
\end{itemize}

\begin{definition}[Bi-accumulation property]
A saddle  $R\in \operatorname{Per}^{s+1} (f)$ is 
{\emph{$\st$-bi-accumulated}} (by homoclinic points) if every component of
$(W^\st_\loc(R,f)\setminus W^\sst_\loc (R,f))$ contains transverse
homoclinic points of $R$. \label{d.biaccumulation}
\end{definition}

We have the following result.

\begin{lemma}\label{l.summary}
Let $f$ be a diffeomorphism with a coindex one cycle
associated to $P$ and $Q$ such that $H(P,f)$ is non-trivial.
Let $\sind (P)=s+1$.
 Then
there is $g$ arbitrarily $C^1$-close to $f$  such that
\begin{itemize}
\item
there is a  saddle $\bar P_g\in \operatorname{Per}_h^{s+1} (H(P_g,g))$ 
that is
$\st$-bi-accumulated and
\item
the diffeomorphism $g$ has a cycle associated to
$\bar P_g$ and $Q_g$.
\end{itemize}
\end{lemma}

\begin{proof}
The lemma 
follows from  \cite{ABCDW07,DG}.
From 
\cite[Proposition 2.3]{ABCDW07}, 
if $H(P,f)$ is non-trivial then there is $g$ arbitrarily
$C^1$-close to $f$ with a cycle associated to $P_g$ and $Q_g$ and
such that $\operatorname{Per}_h^{s+1} (H(P_g,g))$ is infinite.

By \cite[Lemma 3.4]{DG}, 
if the set $\operatorname{Per}_h^{s+1}(H(P,f))$ is infinite
then there is a diffeomorphism $g$ arbitrarily $C^1$-close to $f$  with a cycle
associated to $P_g$ and $Q_g$ and such that
$\operatorname{Per}_h^{s+1} (H(P_g,g))$ contains infinitely many
$\st$-bi-accumulated saddles.
Pick one of these saddles $\bar P_g$ and note that to be bi-accumulated
is a property that
persists under perturbations.
 We can now perturb $g$ to get $h$
with a cycle associated to $\bar P_h$ and $Q_h$, ending the
proof of the  lemma.
\end{proof}




\section{Stabilization of cycles.
Proof of Theorem~\ref{th.proposition}}
\label{s.stabilization}

\subsection{Stabilization of non-twisted cycles}
	\label{ss.nonreversing}
Next
proposition  is the main step to prove the 
stabilization of non-twisted cycles.

\begin{proposition}
\label{p.nonreversingm} Let  $f$ be a diffeomorphism with a
non-twisted cycle associated to saddles $P$ and $Q$ such that
$\sind (P)=\sind (Q)+1$.  Then there is a diffeomorphism $g$
arbitrarily $C^1$-close to $f$ with a partially hyperbolic
saddle-node/flip $S_g$ such that:
\begin{enumerate}
\item \label{i.pos1} $W^{\sst}(S_g,g) \pitchfork W^{\ut}(Q_g,g)\ne\emptyset$,
\item  \label{i.pos2} $W^{\uut}(S_g,g) \pitchfork W^{\st}(P_g,g)\ne\emptyset$,
\item   \label{i.pos3} $W^{\uut}(S_g,g) \cap W^{\sst}(S_g,g)$
contains a point that is not in the orbit of $S_g$
(strong homoclinic intersection),
\item \label{i.pos5}
$W^{\sst}(S_g,g) \cap W^{\ut}(P_g,g)\ne\emptyset$, and
\item \label{i.pos4}
$W^{\uut}(S_g,g) \cap W^{\st}(Q_g,g)\ne\emptyset$.
\end{enumerate}
\end{proposition}

The dynamical configuration in the proposition is depicted in Figure~\ref{f.propnon}.

We postpone the proof of this proposition to Section~\ref{sss.prooftheoremnonreversing}. We now prove (A) in Theorem~\ref{th.proposition}.

\begin{figure}[hbtp]
\centering \scalebox{0.75}{\includegraphics[clip]{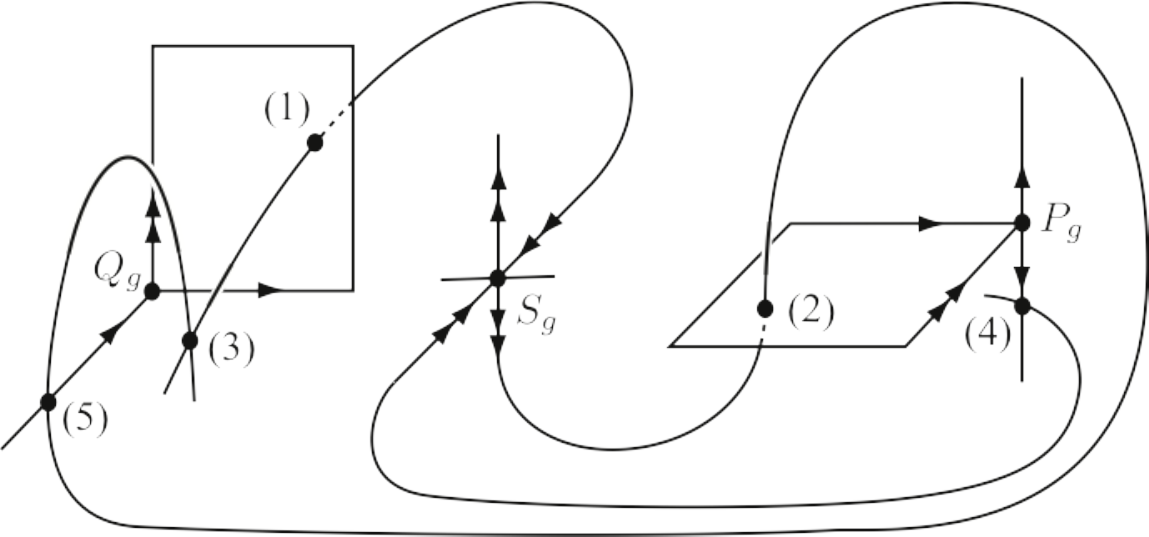}}
\caption{Dynamical configuration in Proposition~\ref{p.nonreversingm}.} \label{f.propnon}
\end{figure}

\subsubsection{Proposition~\ref{p.nonreversingm} implies
(A) in Theorem~\ref{th.proposition}}
\label{sss.prooftheoremnonreversing} 

Note that the transverse intersection conditions immediately imply that
$ \sind (P)=\dim (W^{\sst}(S))+1=s+1$
(condition (1) in Theorem~\ref{t.p.BDblenderb}).
Moreover,
conditions (\ref{i.pos2})--(\ref{i.pos5}) in
Proposition~\ref{p.nonreversingm} imply 
that $S$ and $P$ satisfy (2)--(4) in 
Theorem~\ref{t.p.BDblenderb}.
Thus the diffeomorphism $g$ satisfies all
conditions in Theorem~\ref{t.p.BDblenderb} and 
hence there is $h$
arbitrarily $C^1$-close to $g$ having a robust heterodimensional
cycle associated to $P_h$ and a (transitive) hyperbolic set
$\Gamma_h$ containing a continuation $S_h^+$ of $\st$-index $s$
of $S_g$.

Observe that items (\ref{i.pos1}) and (\ref{i.pos4}) in
Proposition~\ref{p.nonreversingm}
 imply that the saddle $S_h^+$ of $h$ can be chosen such
that
$$
W^{\st}(S^+_h,h) \pitchfork W^{\ut}(Q_h,h)\ne\emptyset 
\quad
\mbox{and} \quad W^{\ut}(S_h^+,h) \cap W^{\st}(Q_h,h)\ne\emptyset.
$$
Thus the saddles $S_h^+$ and $Q_h$ are homoclinically
related and then there is a transitive hyperbolic set
$\Sigma_h$ containing $Q_h$ and $\Gamma_h$. In particular, for
every diffeomorphism $\varphi$ close to $h$ it holds
$W^{\st,\ut}(\Ga_\varphi,\varphi)\subset
W^{\st,\ut}(\Sigma_\varphi,\varphi)$. Thus, by the first step of the proof,
the diffeomorphism $h$ has a robust
cycle associated to $\Sigma_h$ and $P_h$, ending the proof
of  (A) in Theorem~\ref{th.proposition}. \hfill \qed


\subsubsection{Proof of Proposition~\ref{p.nonreversingm}} \label{sss.proofpropositionnonreversingm}
This proposition follows from  Proposition~\ref{p.l.R}. First note that by Lemma~\ref{l.adapted},
after a small perturbation, we can assume that the cycle 
(associated to $P$ and $Q$)
has 
a sequence of adapted homoclinic 
intersections associated to the saddle $P$. Thus applying Proposition~\ref{p.l.R}
we obtain $g$ close to $f$ with a partially hyperbolic saddle-node/flip satisfying
conditions \eqref{i.pos1}, \eqref{i.pos2}, and \eqref{i.pos5} in Proposition~\ref{p.nonreversingm}. It remains to obtain conditions \eqref{i.pos3} ($W^\uut(S_g,g)
\cap W^\sst(S_g,g)$ contains a point that is not in the orbit of $S_g$) and \eqref{i.pos4} ($W^\uut(S_g,g)
\cap W^\st(P_g,g)\ne \emptyset$) in Proposition~\ref{p.nonreversingm}. 
To get these two properties 
we use arguments 
 analogous
 to the ones in Lemmas~\ref{l.adapted} and \ref{l.ifs}.

Since in what follows 
we do not modify the orbits of  $P_g,Q_g$, and $S_g$
let us omit the dependence on $g$.
Note that since 
$W^\uut(S,g)\pitchfork W^\st(P,g)$
(condition \eqref{i.R2} in Proposition~\ref{p.l.R}) we have that
$W^\uut(S,g)$ accumulate to $W^\ut(P,g)$. Since
by condition \eqref{i.R5}
in Proposition~\ref{p.l.R} we have that
 $W^\ut(P,g)\cap W^\st(Q,g)\ne\emptyset$, thus $W^\uut(S,g)$ also 
 accumulates
to $W^\st(Q,g)$. In particular
there are segments of $W^\uut(S,g)$ 
(with disjoint orbits)
arbitrarily close to 
$W^\st_\loc(Q,g)$. We use one of these segments to get 
 $W^\uut(S,h)\cap W^\st(Q,h) \ne \emptyset$
for some $h$ close to $g$ (condition (\ref{i.pos4}) in Proposition~\ref{p.nonreversingm}). 

Moreover, the previous perturbation can be done
in such a way  there are 
segments of $W^\uut(S,h)$ close to $W^\st(Q,h)$ in the ``same side"
of $W^\st(Q,h)$ as $W^\sst(S,h)$. See Figures~\ref{f.perturbation1}
 and \ref{f.perturbation2}. 
Thus modifying the derivative of $Q$
in the central direction  we get that 
$W^\uut(S,h)$ intersects
$W^\sst(S,h)$
(condition (\ref{i.pos3}) in Proposition~\ref{p.nonreversingm}). Note that these perturbations can be done preserving the
saddle-node/flip $S$ and the intersections properties 
 \eqref{i.pos1}, \eqref{i.pos2}, and \eqref{i.pos5} in Proposition~\ref{p.nonreversingm}.  
\hfill \qed


\begin{figure}[hbtp]
\centering \scalebox{0.70}{\includegraphics[clip]{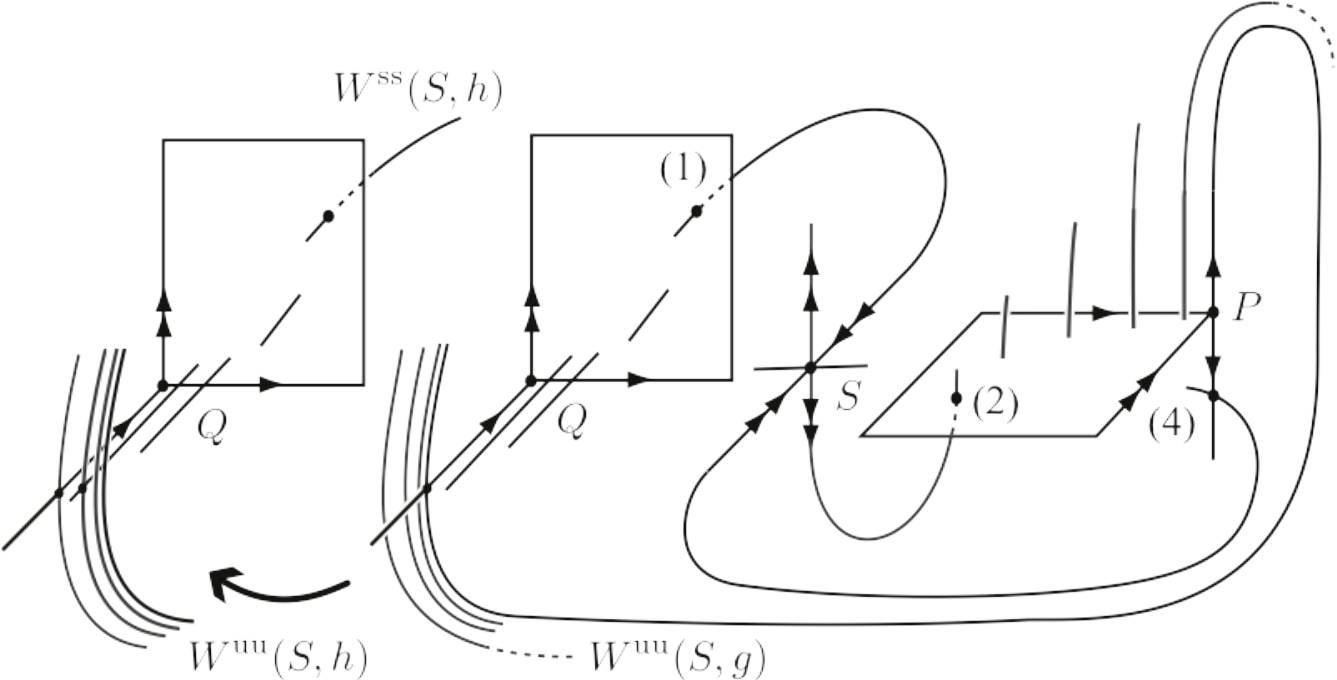}}
\caption{Accumulation of $W^\uut(S)$ to $W^s_\loc(Q)$.} \label{f.perturbation1}
\end{figure}

\begin{figure}[hbtp]
\centering \scalebox{0.70}{\includegraphics[clip]{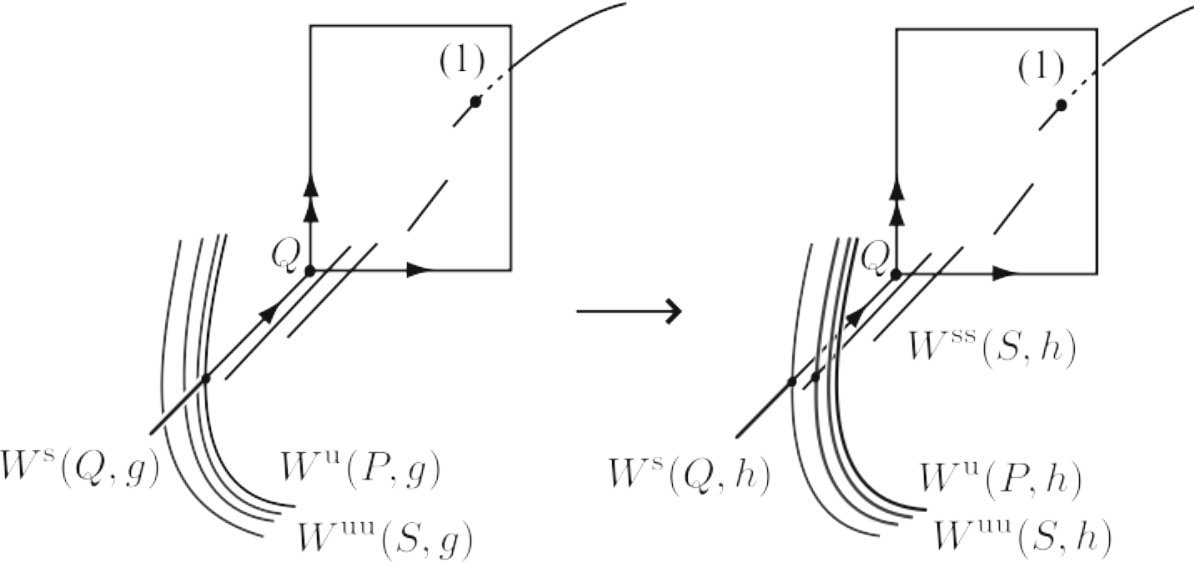}}
\caption{Accumulation of $W^\uut(S)$ to $W^s_\loc(Q)$.} \label{f.perturbation2}
\end{figure}


\subsection{Stabilization of bi-accumulated twisted cycles}
	\label{ss.bitwisted}

In this section we prove item (B) in Theorem~\ref{th.proposition}.

\begin{proposition}[Generation of non-twisted cycles]
\label{p.t.twistednontwisted} Let $f$ be a diffeomorphism 
with a twisted cycle associated to saddles 
$P$ and $Q$ with $\sind (P)=\sind (Q)+1$. Assume that
$P$ is $\st$-bi-accumulated.
Then
there is $g$ arbitrarily $C^1$-close to $f$ with a non-twisted 
 cycle associated
to $Q_g$ and a saddle $R_g$ that is homoclinically related
to $P_g$.
\end{proposition}

Item 
(A) in Theorem~\ref{th.proposition} implies that the cycle associated to 
$R_g$ and $Q_g$ can be stabilized. Since $R_g$ is homoclinically related
to $P_g$, Lemma~\ref{l.homoclinicstabilization}
implies that the cycle associated to $P_f$ and $Q_f$ can also be stabilized.
Thus Proposition~\ref{p.t.twistednontwisted} implies (B) in Theorem~\ref{th.proposition}.

\subsubsection{Proof of  Proposition~\ref{p.t.twistednontwisted}}
\label{ss.cyclesbiaccumulated}
The proposition is an immediate consequence of the following two
lemmas:

\begin{lemma}\label{l.step1}
Under the hypotheses of Proposition~\ref{p.t.twistednontwisted},
 there is $g$
arbitrarily $C^1$-close to $f$ with a 
twisted simple cycle associated to $P$ and $Q$ 
and with an adapted homoclinic point of
$P$.
\end{lemma}

\begin{lemma}\label{l.step2}
Consider a twisted cycle $\scyc
(f,Q,P,\be,\la,-,\pm_2)$, $\la,\be>0$, with an adapted
homoclinic intersection (associated to $P$).
 Then there is $g$
arbitrarily $C^1$-close to $f$ with a saddle $R_g$ such that
 \begin{itemize}
 \item
 $R_g$ is homoclinically related to $P_g$ and
 \item
 $g$ has
 a non-twisted cycle associated to $R_g$ and $Q_g$.
 \end{itemize}
\end{lemma}

\subsubsection{Proof of Lemma~\ref{l.step1}}
We claim  that
(in the coordinates in $U_Q$ in Proposition~\ref{p.simple})
 there are
sequences of points $(x_i)_i$ and $(a_i^s)_i$, $x_i\in \RR$ and
$a_i^s\in \RR^s$, and of disks $\De_i$ of dimension $u$ such that
\begin{itemize}
\item $(a^s_i,x_i,0)\in \De_i$ where  $x_i\to
0^+$ and $a_i^s\to a^s$, and
\item
$\De_i\to \{(a^s,0)\} \times [-1,1]^u$ and $\De_i\subset
W^\ut(P,f)$,
\end{itemize}
 here $(a^s,0,0^u)$ is the heteroclinic intersection
between $W^\ut(P,f)$ and $W^\st(Q,f)$ in (C2) in
Proposition~\ref{p.simple}.


  \begin{figure}[hbtp]
\centering 
\scalebox{0.70}{\includegraphics[clip]{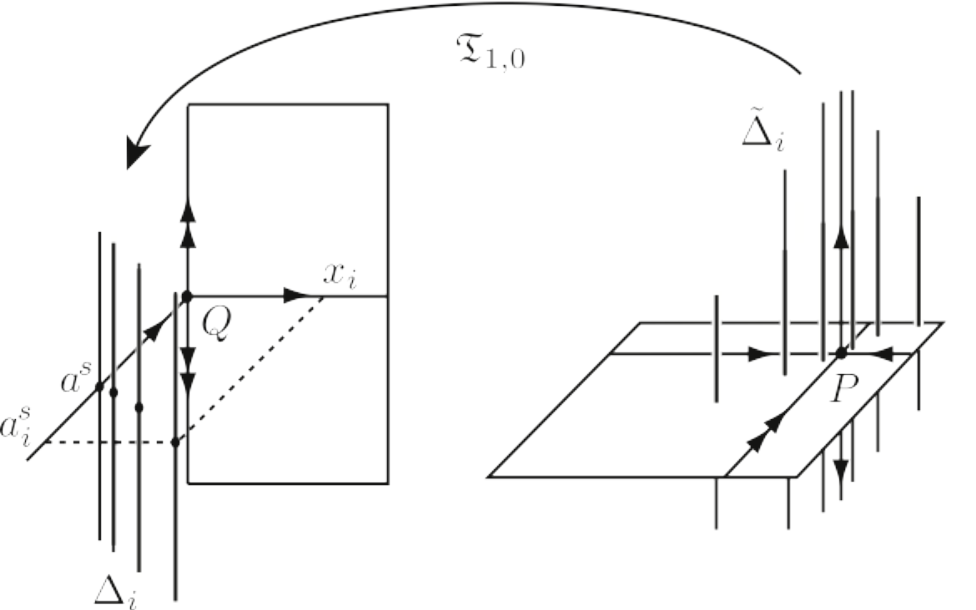}}
\caption{The disks $\De_i$.} \label{f.biaccumulatedtwisted}
\end{figure}
  

To see why this assertion is so just note that,
by the bi-accumulation property, there is
as sequence of unstable disks $\tilde \De_i\subset W^\ut (P,f)$ of dimension $u$
approaching to $W^\ut_\loc (P,f)$ 
from the ``negative side", see Figure~\ref{f.biaccumulatedtwisted}.
Since the cycle is twisted the map 
$\mfT_{1,0}$ reverses the ordering in the central direction. Thus
 these disks are mapped by $\mfT_{1,0}$ into a disks $\De_i$ that approaches
 $(a^s,0,0^u)$ from the ``positive side". 
See Figure~\ref{f.biaccumulatedtwisted}. 
 We need to perform a perturbation
 in order to put these disks in ``vertical" position.
  
  Arguing exactly as in the proof of Lemma~\ref{l.ifs}},
after  an arbitrarily small perturbation we can assume that $\beta$ is such that
$\psi_\be^{k_i}(x_i)=1$ for some
arbitrarily large $i$ and $k_i$. This provides a transverse homoclinic
point of $P$ of the form $(h^s,1,0)$. This follows from (F) in Proposition~\ref{p.dictionary}. Note that this perturbation can be done preserving
the cycle between $P$ and $Q$.

Finally, using
this transverse homoclinic point and
after an arbitrarily small perturbation, we
get the simple cycle with an adapted homoclinic intersection associated
to $P$ and $Q$ (the argument is similar to the one
in Lemma~\ref{l.reduction}.)
\hfill \qed

\subsection{Proof of Lemma~\ref{l.step2}}
The lemma follows arguing as in \cite[Lemma 3.13]{BDijmj}
and using Proposition~\ref{p.dictionary}.  
Note that we can assume (after a small modification of $\be$ and
$\la$) that $\be^{-m}=\la^k$. Noting that the cycle is twisted
(i.e., $\theta_{1,t}(x)=t-x$) we have that this equality implies that
$$
\Ga^{m,k}_0(1)=
\psi_\be^m \circ \theta_{1,0} 
\circ \phi_\la^k \circ \theta_2(1)=
\psi_\be^m \big(- \phi_\la^k (-1)\big)=1.
$$
In this case we also have,
$$
(\Ga_0^{m,k})'(1)=
(\psi_\be^m)'  \big(- \phi_\la^k (-1)\big)\,
 (\phi_\la^k)' (-1)=\pm \be^m\,\la^k=\pm 1.
$$
Thus  modifying the central derivatives at $P$ and $Q$, we can assume
that the cycle is semi-simple with 
central maps $\tilde \psi_\be$ and $\tilde \phi_\la$
such that there are 
large $m, k,$ and $\ell$, with $\ell>>k$, satisfying

\begin{equation}\label{e.periodicpoint}
\tilde\psi^m_\be \big(-\tilde \phi_\la^{k}(-1)+
\tilde \phi_\la^{\ell}(-1)\big)=1
\end{equation}
and
\begin{equation}
\label{e.derivadas1}
|\big(\tilde \psi_\beta^m\big)'  
\big( - \tilde \phi_\la^k (-1) + \tilde \phi_\la^\ell (-1)
\big)\,
\big( \tilde \phi_\la^k \big)' (-1)|<1.
\end{equation}
For that note that $\tilde \phi_\la^\ell(-1)$ is arbitrarily small in comparison  
with $\tilde \phi_\la^k(-1)$.

We now select the parameter 
$t=\tilde \phi^{\ell}(-1)<0$.
By equation~\eqref{e.periodicpoint} one has
\[
\begin{split}
\tilde \Ga_t^{m,k}(1) &= \tilde \psi^m_\beta \circ 
\theta_{1,t}\circ \tilde \phi_\la^k
\circ \theta_2(1)=
\tilde \psi^m_\beta \circ \theta_{1,t}\big( \tilde \phi_\la^k(-1)\big)\\
&=
\tilde\psi^m_\be \big(-\tilde \phi_\la^{k}(-1)+
\tilde \phi_\la^{\ell}(-1)\big)=1.
\end{split}
\]
Let $R=(r^s,1,r^u)\in U_Q$ be the saddle of $f_t$ associated to $1$ and the itinerary $(m,k)$
given by (A) in Proposition~\ref{p.dictionary}. 
Note that
\[
\begin{split}
|(\tilde \Ga_t^{m,k})' (1)|&=
|\big(\tilde \psi_\beta^m\big)'  
\big( 
\theta_{1,t} 
\big(\tilde \phi_\la^k ( \theta_2(1) )
\big)
\big)\,
\big( \tilde \phi_\la^k \big)' (\theta_2(1))|\\
&=
|\big(\tilde \psi_\beta^m\big)'  
\big( \tilde \phi_\la^k (-1) + \tilde \phi_\la^\ell (-1)
\big)\,
\big( \tilde \phi_\la^k \big)' (-1)|<1,
\end{split}
\]
where the inequality follows from \eqref{e.derivadas1}.
By (A) in Proposition~\ref{p.dictionary} the saddle $R$ 
has index $s+1$.
Indeed, since 
 $\theta_{1,t}(x)=-x+t$, the central multiplier of $R$ is
 positive  if
 $\theta_2$ reverses the orientation and negative  otherwise.
 
 We claim that the saddle $R$ is homoclinically related to $P$ and has a
 cycle associated to $Q$. Note that $W^\uut(R,f_t)=W^\ut(R,f_t)$.  
 
By equation~\eqref{e.initemA} in Proposition~\ref{p.dictionary}
we have that
\begin{equation}\label{e.cyclehomoclinic}
W^\st(R,f_t) \pitchfork W^\ut(Q,f_t)\ne\emptyset
\quad
\mbox{and}
\quad
W^\ut(R,f_t) \pitchfork W^\st(P,f_t)\ne\emptyset.
\end{equation}

From the existence of an adapted
homoclinic intersection and 
item (E)(1) in
Proposition~\ref{p.dictionary}:
\begin{itemize}
\item
 $H=(h^s,1,0)$ is a transverse
homoclinic point of $P$, 
\item
$\{(h^s,1)\}\times [-1,1]^u \subset
W^\ut (P,f_t) \cap U_Q$, and 
\item 
$\{[-1,1]^s\times \{(1,r^u)\}\subset W^{\sst}(R,f_t)$.
\end{itemize}
This implies that $W^\ut(P,f_t)\pitchfork W^\st(R,f_t)$.
Thus, by the second part
of \eqref{e.cyclehomoclinic},  the saddles $P$ and $R$ are
homoclinically related for $f_t$. 

To get cycle associated to $R$ and $Q$
note that the choice of $t$ implies that
$$
\theta_{1,t}\circ 
\tilde \phi_\la^\ell \circ \theta_2(1)=
-\tilde\phi_\la^\ell (-1)+ t=0.
$$
Since $R=(r^s,1,r^u)$, condition
(D) in Proposition~\ref{p.dictionary} implies
that $W^\ut(R,f_t)\cap W^\st(Q,f_t)\ne\emptyset$.
Thus by the first part 
of \eqref{e.cyclehomoclinic} the diffeomorphism 
$f_t$ has a cycle associated to $R$ and $Q$.

We claim that this cycle is non-twisted.
If $\theta_2$ reverses the orientation then the central
multiplier of $R$ is negative and the cycle is non-twisted.
Otherwise, we have a cycle whose central ``unfolding map" 
is obtained considering the composition 
$\theta_{1,t} \circ \tilde \phi_\la \circ 
\theta_2$. This map preserves the central orientation:
just note that $\theta_{1,t}$ and $\theta_2$ both reverse the orientation
 and $\tilde\phi_\la$ preserves this orientation
(recall that $\la>0$). This completes the proof of the lemma.
\hfill
\qed




\section{Proof of Theorems~\ref{t.homoclinic} and \ref{t.complexornontwisted}}
\label{s.proofoftheoremsAB}

\subsection{Proof of Theorem~\ref{t.complexornontwisted}.}
\label{s.proofoftheorem}
Note that (B) in Theorem~\ref{t.complexornontwisted} is an 
immediate consequence from (A) in Theorem~\ref{th.proposition}.

To prove item (A) let us assume that, for instance, the saddle $P$ has 
non-real central multipliers.
By  Theorem~\ref{t.complex} (see also Remark~\ref{r.non-trivial})
there is $g$ close to $f$ 
having 
 saddles
$P'_g$ and $Q'_g$
such that
\begin{itemize}
\item there is a cycle with real central multipliers
associated to $P_g'$ and $Q_g'$,
\item 
$P_g'$ and $Q_g'$ are 
 homoclinically related to 
 $P_g$ and $Q_g$,
 \item
 the homoclinic class
of $P'_g$ is non-trivial (note that we may have $Q_g'=Q_g$ and a
trivial homoclinic class $H(Q_g,g)$).
\end{itemize}
 By Lemma~\ref{l.homoclinicstabilization}
it is enough to prove that this new cycle can be stabilized.

If the cycle associated to $P'_g$ and $Q'_g$ is non-twisted the stabilization follows
from (A) in Theorem~\ref{th.proposition}. 
Otherwise, if the cycle is twisted, by Lemma~\ref{l.summary} 
there is a diffeomorphism $h$ close to $g$ having a
saddle
 $\bar P_h$ 
 such that
 
 \begin{itemize} 
 \item $\bar P_h$ is
 homoclinically related to $P'_h$ and has  the bi-accumulation property,
 \item
 there is a cycle associated to $Q_h'$ and $\bar P_h$. Note that this cycle
 has real central multipliers.
 \end{itemize}
 As above, it is enough to prove that this cycle can be stabilized. 
 The stabilization of this cycle follows from Theorem~\ref{th.proposition}. This ends the proof
 of the theorem. \hfill \qed

\subsection{Proof of Theorem~\ref{t.homoclinic}} \label{ss.proofofhomoclinic}
By
Theorem~\ref{t.complex} and  
Lemma~\ref{l.homoclinicstabilization}  
we can assume that the cycle associated to
the saddles $P$ and $Q$ 
has real central multipliers and that, for instance, 
the homoclinic class of $P$ is non-trivial.
If the cycle is non-twisted the result follows
from (A) in Theorem~\ref{th.proposition}.

Otherwise, if
the cycle is twisted,
arguing as in the proof of Theorem~\ref{t.complexornontwisted}, 
there is a diffeomorphism $g$ close to $f$ having a cycle 
associated to $Q_g$ and to 
a saddle $\bar P_g$ that is
homoclinically related to $P_g$ and satisfies the
$\st$-bi-accumulation property.
By (B) in Theorem~\ref{th.proposition} this cycle can be stabilized.
Since
 $\bar P_g$ is homoclinically related to $P_g$
  the initial cycle
also can be stabilized, ending the proof of the theorem. 
\hfill \qed

\subsection{Proof of Corollary~\ref{c.corol}} 
This result follows immediately from Theorem~\ref{t.homoclinic}
considering the following perturbation of  the initial cycle. 
First, we preserve one of the
heterocinic orbits in $W^\ut(P,f)\cap W^\st(Q,f)$.
We can also assume that $W^\st(P,f)$ transversely intersects 
$W^\ut(Q,f)$ and thus accumulates to $W^\st(Q,f)$. We can now
use the second
heteroclinic orbit in 
$W^\ut(P,f)\cap W^\st(Q,f)$ to get a transverse homoclinic point of $P$.
In this way we obtain a cycle satisfying Theorem~\ref{t.homoclinic}.

\section*{Acknowledgments}
The authors would like to express their gratitude to H. Kokubu, M. C. Li, and
M. Tsujii for their hospitality and financial support during their
visits to RIMS (Japan)  and NCTU (Taiwan) where a
substantial part of this paper was developed. 
This paper is also partially supported by CNPq, FAPERJ, and Pronex
(Brazil), ``Brazil-France Cooperation in Mathematics" 

The authors also thank  S. Crovisier, K. Shinohara, and  T. Soma for useful conversations in this 
subject.


\end{document}